\documentclass[11pt]{amsart}
\usepackage{amsmath}
\usepackage{amssymb}
\usepackage{amsfonts}
\usepackage{amsthm}
\usepackage{enumerate}
\usepackage[mathscr]{eucal}
\usepackage{verbatim}
\usepackage{amsthm}
\usepackage{amscd}
\usepackage[mathscr]{eucal}
\usepackage{appendix}
\usepackage{tikz}

\numberwithin{equation}{section}

\newtheorem{innercustomgeneric}{\customgenericname}
\providecommand{\customgenericname}{}

\newcommand{\newcustomtheorem}[2]{\newenvironment{#1}[1]
  {\renewcommand\customgenericname{#2}
   \renewcommand\theinnercustomgeneric{##1}\innercustomgeneric}{\endinnercustomgeneric}}

\newcustomtheorem{customthm}{Theorem}

\newcommand{\newcustomlemma}[2]{\newenvironment{#1}[1]
  {\renewcommand\customgenericname{#2}
   \renewcommand\theinnercustomgeneric{##1} \innercustomgeneric}{\endinnercustomgeneric}}

\newcustomlemma{customlemma}{Lemma}

\newcustomlemma{customproposition}{Proposition}

\newcustomlemma{customclaim}{Claim}

\theoremstyle{plain}
\newtheorem{theorem}{Theorem}[section]
\newtheorem{lemma}[theorem]{Lemma}

\newtheorem{proposition}[theorem]{Proposition}

\theoremstyle{remark}
\newtheorem{remark}{Remark}

\theoremstyle{definition}

\newcommand{\q}{\quad}
\newcommand{\qq}{\qquad}

\newcommand{\ga}{\gamma}
\newcommand{\Ga}{\Gamma}
\newcommand{\om}{\omega}
\newcommand{\Om}{\Omega}
\newcommand{\la}{\lambda}
\newcommand{\La}{\Lambda}

\newcommand{\ep}{\epsilon}

\newcommand{\bbz}{\mathbb{Z}}

\newcommand{\bbzn}{\mathbb Z^n}

\newcommand{\bbr}{\mathbb{R}}

\newcommand{\bbrn}{\mathbb R^n}
\newcommand{\rn}{\mathbb{R}^n}

\newcommand{\bbn}{\mathbb{N}}

\newcommand{\OO}{\mathcal{O}}
\newcommand{\HH}{\mathcal{H}}

\newcommand{\UU}{\mathcal{U}}
\newcommand{\WW}{\mathcal{W}}
\newcommand{\II}{\mathcal{I}}

\newcommand{\TT}{\mathcal{T}}
\newcommand{\LL}{\mathcal{L}}
\newcommand{\PP}{\mathcal{P}}
\newcommand{\BB}{\mathscr{B}}
\newcommand{\DD}{\mathscr{D}}

\newcommand{\kkk}{\vec{\boldsymbol{k}}}
\newcommand{\xxxi}{\vec{\boldsymbol{\xi}\;}}
\newcommand{\xxx}{\vec{\boldsymbol{x}}}
\newcommand{\yyy}{\vec{\boldsymbol{y}}}
\newcommand{\zzz}{\vec{\boldsymbol{z}}}

\newcommand{\uuu}{\vec{\boldsymbol{u}}}
\newcommand{\GGG}{\vec{\boldsymbol{G}}}
\def\000{\vec{\boldsymbol{0}}}

\def\ii{{\mathrm{i}}}

\newcommand{\supp}{\mathrm{supp}}

\newcommand{\f}{\frac}

\newcommand{\nf}{\infty}

\newcommand{\tf}{\tfrac}
\newcommand{\wh}{\widehat}
\newcommand{\wt}{\widetilde}

\newcounter{question}

\newcommand{\bpf}{\begin{proof}}

\newcommand{\epf}{\end{proof}}

\newcommand{\gap}{7}
\newcommand{\gaptt}{6}
\newcommand{\gapt}{7}

\allowdisplaybreaks

\makeindex         

\begin{document}

\author{Loukas Grafakos}
\address{Department of Mathematics, University of Missouri, Columbia, MO 65211, USA} 
\email{grafakosl@missouri.edu}

\author{Danqing He}
\address{School of Mathematical Sciences,
Fudan University, People's Republic of China}
\email{hedanqing@fudan.edu.cn}

\author{Petr Honz\'ik}
\address{Department of Mathematics,
Charles University, 116 36 Praha 1, Czech Republic}
\email{honzik@gmail.com}

\author{Bae Jun Park}
\address{B. Park, Department of Mathematics, Sungkyunkwan University, Suwon 16419, Republic of Korea}
\email{bpark43@skku.edu}

\thanks{The research of L. Grafakos is partially supported
by the Simons Foundation Grant 624733 and by the Simons Fellows award 819503.
 D. He is supported by  National Key R$\&$D Program of China (No. 2021YFA1002500), NNSF of China (No. 12161141014),  and Natural Science Foundation of Shanghai (No. 22ZR1404900). P. Honz\'\i{}k was supported by the grant GA\v{C}R P201/21-01976S.
B. Park is supported by NRF grant 2022R1F1A1063637.}

\title{Multilinear Rough Singular Integral operators}

\begin{abstract} 
We study $m$-linear homogeneous rough singular integral operators $\LL_{\Omega}$ associated with  integrable functions $\Omega$ on $\mathbb{S}^{mn-1}$ with mean value zero. We prove 
boundedness for $\LL_{\Omega}$ from 
$L^{p_1}\times \cdots \times L^{p_m}$ to $L^p$  when 
$1<p_1,\dots, p_m<\nf$ and $1/p=1/p_1+\cdots +1/p_m$ in the largest possible open set of 
exponents when $\Omega \in L^q(\mathbb S^{mn-1})$ and $q\ge 2$. 
This set can be described by  a convex polyhedron in $\mathbb R^m$. 
\end{abstract}

\maketitle


\section{Introduction}

 Let $\Omega$ be an integrable function on the unit sphere $\mathbb{S}^{n-1}$ with mean value $0$. 
The rough singular integral operator associated with $\Omega$ is defined by
$$
\LL_{\Omega}f(x):=\mathrm{p.v.}\int_{\bbrn}\frac{\Omega(y/|y|)}{|y|^n}f(x-y)dy
$$
initially  for $f$ in the Schwartz class $\mathscr{S}(\bbrn)$.
 
 Calder\'on and Zygmund \cite{Ca_Zy1956} proved that if $\Omega\in L\log L(\mathbb{S}^{n-1})$, then $\LL_{\Omega}$ is bounded on $L^p(\bbrn)$ for all $1<p<\infty$. This result was improved by Coifman and Weiss \cite{Co_We1977} who replaced the condition $\Omega\in L\log L(\mathbb{S}^{n-1})$ by the less restrictive condition $\Omega\in H^1(\mathbb{S}^{n-1})$. The same conclusion was also obtained independently by Connett \cite{Co1979}.
 In the two dimensional case $n=2$, the weak type $(1,1)$ of $\LL_{\Omega}$ was established by Christ \cite{Ch1988} and independently by Hofmann \cite{Ho1988} for $\Omega\in L^q(\mathbb{S}^1)$, $1<q\le \infty$, and by Christ and Rubio de Francia \cite{Ch_Ru1988} for $\Omega\in L\log L(\mathbb{S}^1)$. These results were extended to all dimensions by Seeger \cite{Se1996}.
 
\hfill
 
In this paper we focus on analogous questions for $m$-linear singular integral operators. 
Throughout this paper we fix $m$ to be an integer greater or equal to $2$.
 Let $\Omega$ be an integrable function on
$ \mathbb{S}^{mn-1}  $ with mean value zero, and we introduce a kernel $K$ by setting  
\begin{equation*}
K(\yyy):=\frac{\Omega(\yyy')}{|\yyy|^{mn}}, \qquad \vec y \neq 0, 
\end{equation*} 
where $\yyy':=\yyy/|\yyy\,|\in \mathbb{S}^{mn-1}$. 
Then the multilinear singular integral operator associated with $\Om$ is defined as follows:   
\begin{equation*}
\LL_{\Om}\big(f_1,\dots,f_m\big)(x):=\mathrm{p.v.} \int_{(\bbrn)^m}{K(\yyy)\prod_{j=1}^{m}f_j(x-y_j)}~d\,\yyy
\end{equation*} 
for Schwartz functions $f_1,\dots,f_m$ on $\bbrn$, where $x\in \bbrn$ and $\yyy:=(y_1,\dots,y_m)\in (\bbrn)^m$.

The first important result concerning bilinear $(m=2)$ rough singular integrals appeared in the work of Coifman and Meyer \cite{Co_Me1975} who obtained an estimate for $\LL_{\Omega}$ when $\Omega$ possesses some smoothness. These authors actually showed that  if $\Omega$ is a function of bounded variation on the circle $\mathbb{S}^1$,
then the corresponding bilinear operator $\LL_{\Omega}$ is bounded from $L^{p_1}(\bbr)\times L^{p_2}(\bbr)$ to $L^p(\bbr)$ when $1<p_1,p_2,p<\infty$ and $1/p_1+1/p_2=1/p$. Later, 
for general dimensions $n\ge 1$ and all $m\ge 2$, Grafakos and Torres \cite{Gr_To2002} established the $L^{p_1}(\bbrn)\times \cdots \times L^{p_m}(\bbrn)\to L^p(\bbrn)$ boundedness of $\LL_{\Omega}$ for all $1<p_1,\dots, p_m<\infty$ with $1/p=1/p_1+\cdots+1/p_m$ when $\Omega$ is a Lipschitz function on $\mathbb{S}^{mn-1}$. 
The case of rough $\Om$ was not really addressed until the work of Grafakos, He, and Honz\'ik \cite{Gr_He_Ho2018} who proved   bilinear estimates in the full range $1<p_1,p_2<\nf$  under the condition $\Omega\in L^{\infty}(\mathbb{S}^{2n-1})$. These authors also showed that $\LL_{\Omega}$ maps $L^2(\bbrn)\times L^2(\bbrn)$ into $L^1(\bbrn)$ if $\Omega$ is merely an $L^2$ function on $\mathbb{S}^{2n-1}$. 
This initial $L^2(\bbrn)\times L^2(\bbrn)\to L^1(\bbrn)$ estimate was refined by Grafakos, He, and Slav\'ikov\'a \cite{Gr_He_Sl2020}   replacing $\Omega\in L^2(\mathbb{S}^{2n-1})$ by $\Omega\in L^q(\mathbb{S}^{2n-1})$ for $q>4/3$.
 Recently, He and Park \cite{He_Park2021} proved more points of boundedness  for the bilinear rough singular integral operators in the  range $1<p_1,p_2\le \infty$ except the endpoint $p_1=p_2=\infty$.
 \begin{customthm}{A}\cite{He_Park2021}\label{thma}
 Let $1<p_1,p_2\le\infty$ and $1/2< p<\infty$ with $1/p=1/p_1+1/p_2$.
Suppose that 
\begin{equation}\label{bilinearqcondition}
\max{\Big( \frac{4}{3},\frac{p}{2p-1} \Big)}<q\le \infty 
\end{equation}
and $\Omega \in L^q (\mathbb S^{2n-1})$ with $\int_{\mathbb S^{2n-1}} \Omega \, d\sigma=0$. Then the estimate 
\begin{equation}\label{bilinearest}
\big\Vert \LL_{\Omega}\big\Vert_{L^{p_1}\times L^{p_2}\to L^p}\lesssim \Vert \Omega\Vert_{L^q(\mathbb{S}^{2n-1})} 
\end{equation}
is valid. 
 \end{customthm}

In this paper, we will study a multilinear analogue of Theorem \ref{thma}. Such an extension is 
naturally more complicated combinatorially, but also presents additional difficulties 
as $L^2\times \cdots \times L^2$ maps into the nonlocally convex space $L^{2/m}$ when $m\ge 3$. 

In order to present our main results,   we first introduce some notation.
Let $J_m:=\{1,\dots,m\}$.
For $0<s<1$ and any subsets $J\subseteq  J_m$, we let
\begin{equation*}
\HH_J^m(s):=\Big\{(t_1,\dots,t_m)\in (0,1)^m: \sum_{j\in J}(s-t_j)>-(1-s) \Big\},
\end{equation*}
\begin{equation*}
\OO_J^m(s):=\Big\{(t_1,\dots,t_m)\in (0,1)^m: \sum_{j\in J}(s-t_j)<-(1-s) \Big\}
\end{equation*}
and we define
\begin{equation}\label{defhh}
\HH^m(s):=\bigcap_{J\subseteq J_m}\HH_J^m(s).
\end{equation}
We observe that 
$$\mathcal{H}^m(s_1)\subset \mathcal{H}^m(s_2)\subset (0,1)^m \q \text{ for }~s_1<s_2$$
and $\lim_{s\nearrow 1}\mathcal{H}^m(s)=(0,1)^m$.
See Figure \ref{fighs} for the shape of $\mathcal{H}^3(s)$ in the trilinear setting.

 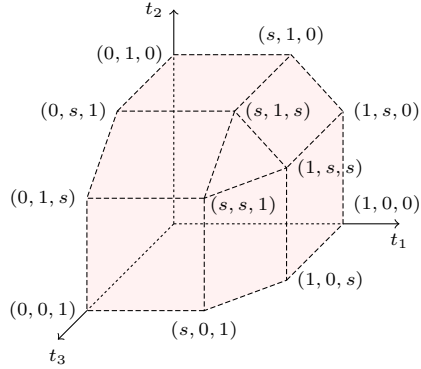
\begin{figure}[h]
\begin{tikzpicture}

\path[fill=red!5] (0,0,3)--(1.56,0,3)--(2.25,0,1.95)--(2.25,0,0)--(2.25,1.495,0)--(1.56,2.25,0)--(0,2.25,0)--(0,2.25,1.95)--(0,1.495,3)--(0,0,3);

\draw[dash pattern= { on 2pt off 1pt}](1.56,1.495,3)--(1.56,2.25,1.95)--(2.25,1.495,1.95)--(1.56,1.495,3);
\draw[dash pattern= { on 2pt off 1pt}] (1.56,1.495,3)--(1.56,0,3)--(2.25,0,1.95)--(2.25,1.495,1.95);
\draw[dash pattern= { on 2pt off 1pt}] (1.56,2.25,1.95)--(1.56,2.25,0)--(2.25,1.495,0)--(2.25,1.495,1.95);
\draw[dash pattern= { on 2pt off 1pt}] (1.56,1.495,3)--(0,1.495,3)--(0,2.25,1.95)--(1.56,2.25,1.95);
\draw[dash pattern= { on 2pt off 1pt}] (2.25,0,1.95)--(2.25,0,0)--(2.25,1.495,0);
\draw[dash pattern= { on 2pt off 1pt}] (0,2.25,1.95)--(0,2.25,0)--(1.56,2.25,0);
\draw[dash pattern= { on 2pt off 1pt}] (1.56,0,3)--(0,0,3)--(0,1.495,3);

\draw[dash pattern= { on 1pt off 1pt}] (0,0,0)--(2.25,0,0);
\draw[dash pattern= { on 1pt off 1pt}] (0,0,0)--(0,2.25,0);
\draw[dash pattern= { on 1pt off 1pt}] (0,0,0)--(0,0,3);
\draw [->] (0,0,3)--(0,0,4);
\draw [->] (0,2.25,0)--(0,2.85,0);
\draw [->] (2.25,0,0)--(3,0,0);

\node [below] at (3,0,0) {\tiny$t_1$};
\node [left] at (0,2.85,0) {\tiny$t_2$};
\node [below] at (0,0,4) {\tiny$t_3$};

\node[above right] at (2.25,0,0) {\tiny$(1,0,0)$};

\node[left] at (0,2.25,0) {\tiny$(0,1,0)$};

\node[left] at (0,0,3) {\tiny$(0,0,1)$};

\node[right] at (2.25,0,1.95) {\tiny$(1,0,s)$};

\node[right] at (2.25,1.495,1.95) {\tiny$(1,s,s)$};

\node[below] at (1.56,0,3) {\tiny$(s,0,1)$};

\node[right] at (1.5,1.4,3) {\tiny$(s,s,1)$};

\node[left] at (0,1.495,3) {\tiny$(0,1,s)$};

\node[left] at (0,2.25,1.95) {\tiny$(0,s,1)$};

\node[above] at (1.56,2.25,0) {\tiny$(s,1,0)$};

\node[right] at (1.56,2.25,1.95) {\tiny$(s,1,s)$};

\node[right] at (2.25,1.495,0) {\tiny$(1,s,0)$};

\end{tikzpicture}
\caption{The region $\mathcal{H}^3(s)$ }\label{fighs}
\end{figure}
Then the first main result of this paper as   follows: 
\begin{theorem}\label{mainthm}
Let $1/2\le s<1$ and $\Omega\in L^{\frac{1}{1-s}}(\mathbb{S}^{mn-1})$ with  $\int_{\mathbb S^{mn-1}} \Omega \, d\sigma=0$.
Suppose that $1<p_1,\dots,p_m<\infty$ and $0<p<1/m$ satisfy $1/p=1/p_1+\cdots+1/p_m$ and
\begin{equation}\label{1pjest}
(1/p_1,\dots,1/p_m)\in \mathcal{H}^m(s).
\end{equation} 
Then we have
\begin{equation}\label{mainthmest}
\big\Vert \LL_{\Om}\big\Vert_{L^{p_1}\times\cdots\times L^{p_m}\to L^{p}}\lesssim\Vert \Om\Vert_{L^{\frac{1}{1-s}}(\mathbb{S}^{mn-1})}.
\end{equation}
\end{theorem}

It is natural to ask for the optimal range of indices $p_j$ in \eqref{1pjest} for which \eqref{mainthmest} holds.  Our second main result below provides the necessity of the condition \eqref{1pjest}.
We note that the intersection in \eqref{defhh} can be actually taken over $J\subseteq J_m$ with $|J|\ge 2$ as the inequality $\sum_{j\in J}(s-t_j)>-(1-s)$ is trivial for $0<t_j<1$, $j=1,\dots,m$, if $|J|\le 1$.
\begin{theorem}\label{mainthm2}
Let $0<s<1$. Suppose that $1<p_1,\dots,p_m<\infty$ and $1/m<p<\infty$ satisfy 
\begin{equation}\label{oocondition}
(1/p_1,\dots,1/p_m)\in    \bigcup_{J\subset J_m : |J|\ge 2}\OO^m_J(s)
\end{equation}
and  $1/p=1/p_1+\cdots+1/p_m$.
Then there exists $\Omega\in L^{\frac{1}{1-s}}(\mathbb{S}^{mn-1})$ with  $\int_{\mathbb S^{mn-1}} \Omega \, d\sigma=0$ such that
  estimate \eqref{mainthmest} does not hold.
  In particular, for $q<\tf{2(m-1)}{m}$, there exists a function $\Omega\in L^{q}(\mathbb{S}^{mn-1})$  with  $\int_{\mathbb S^{mn-1}} \Omega \, d\sigma=0$ such that
  $\mathcal L_\Om$ is unbounded from $L^2\times \cdots \times L^2$ to $L^{2/m}$.
\end{theorem}

Thus, combining Theorems~\ref{mainthm} and~\ref{mainthm2}  we obtain  that 
$\mathcal{H}^m(1/q')$ is the largest open set of indices  
$(1/p_1,\dots , 1/p_m)$ for  which boundedness holds for 
$ \LL_{\Omega}$ when $\Om\in L^q(\mathbb S^{mn-1})$ and $q\ge 2$.  (Here $q'$ is the dual index of $q$). 

\begin{remark}
When $m=2$,   condition \eqref{oocondition} is equivalent to $s+1<1/p$ and this implies
 that if $\Vert \LL_{\Omega}\Vert_{L^{p_1}\times L^{p_2}\to L^p}\lesssim \Vert \Omega\Vert_{L^q(\mathbb{S}^{2n-1})}$ holds for $1<p_1,p_2<\infty$ with $1/p_1+1/p_2=1/p$, then we must have 
$\frac{p}{2p-1}\le q$ as $q=\frac{1}{1-s}$. This clearly indicates the necessity of one part of the condition \eqref{bilinearqcondition} in Theorem \ref{thma}.
\end{remark}

\begin{remark}
It is still unknown whether the bilinear estimate \eqref{bilinearest} holds when $q=\frac{p}{2p-1}$ in Theorem \ref{thma}.
In general, we have no conclusion in Theorem \ref{mainthm} when
$$\sum_{j\in J_0}(s-1/p_j)=-(1-s)~\text{  for some }~ J_0\subseteq J_m$$ and
$$(1/p_1,\dots,1/p_m)\in \Big( \bigcup_{J\subset J_m : |J|\ge 2}\OO_J(s) \Big)^c.$$
\end{remark}

\begin{remark}
It is proved in \cite{Paper1} that if  $\Omega\in L^q(\mathbb{S}^{mn-1})$ for $q>\frac{2m}{m+1}$ and $\int_{\mathbb{S}^{mn-1}}\Omega d\sigma=0$, then  
\[
\Vert \LL_{\Omega}\Vert_{L^2\times \cdots\times L^2\to L^{2/m}}\lesssim \Vert \Omega\Vert_{L^q(\mathbb{S}^{mn-1})}.
\]
 In view of this, one might think that 
Theorem \ref{mainthm}  also holds for certain $s<\frac 12$, and possibly for all $s>\f{m-1}{2m}$ or even for all $s>0$.
\end{remark}

To summarize, although the case $\Om\in L^q(\mathbb{S}^{mn-1})$ with $q\ge 2$ is   resolved in this paper, only partial results are currently known in the case $q<2$ as it presents several challenges. 

The proof of Theorem \ref{mainthm} is based on the dyadic decomposition of Duoandikoetxea and Rubio de Francia \cite{Du_Ru1986} and  on its  $m$-linear adaptation contained in some of the aforementioned references.  The main idea is as follows: 
We express the operator $\LL_{\Omega}$ as
$ \sum_{\mu\in\bbz}\LL_{\mu}$
where $\Vert \LL_{\mu}\Vert_{L^{p_1}\times \cdots \times L^{p_m}\to L^p}\lesssim 2^{\delta_0\mu}\Vert \Omega\Vert_{L^q}$   for all $1<q<\infty$ and some $\delta_0>0$, depending on $q$.
As the series   is summable when $\mu<0$, we  focus on obtaining an good decay when $\mu\to +\infty$. 
Such an estimate is stated in \eqref{mainmaingoal} below. In order to obtain this estimate, we apply multilinear interpolation   between  an initial $L^2\times \cdots\times L^2\to L^{2/m}$ estimate with exponential decay $2^{-\wt{\delta}}\mu$ for some fixed number $\wt{\delta}>0$ 
and  general $L^{p_1}\times\cdots\times L^{p_m}\to L^p$ estimates with arbitrarily slow growth in Proposition \ref{mainproposition}.

The arbitrarily slow growth estimate obtained in Proposition \ref{mainproposition} is actually  the main 
contribution   of  this paper.
Let us explain our strategy in more details. 
Unlike the bilinear case, we are not able to obtain estimates for the local $L^2$ cases (namely $p_1,p_2,p'\in[2,\nf)$) from the initial estimates by duality. To overcome this obstacle, when $q=2$, we refine the column-argument developed in \cite{Paper1} to obtain the  estimate in the upper $L^2$ case (i.e., $p_1,p_2\in [2,\nf)$). This combined with a modified Calder\'on-Zygmund argument developed in \cite{He_Park2021} yields the desired range for $q=2$. For $q=\tf{2m}{m-1}$, based on the {\color{purple}estimate }
for $q=2$, a simple geometric observation about the range of indices leads to the estimates in the upper $L^{\tfrac{2m}{m+1}}$ case, and hence the full desired range by the modified Calder\'on-Zygmund argument. Repeating this process gives Proposition \ref{mainproposition} for all $q\in[2,\nf)$. We remark that this induction argument still holds when $q<2$, but the initial case $q=2$ stops us from obtaining Theorem~\ref{mainthm} for this range of $q$.

As far as the proof of Theorem \ref{mainthm2} is concerned, we adapt an idea appearing in some primordial form in \cite{counterex}, whose adaption can be found in \cite{Gr_He_Sl2019}. 
\hfill

{\bf Organization.}
This paper is organized as follows. We   first give the proof of Theorem \ref{mainthm2} by constructing counterexamples in Section \ref{pfthm2}.
We   reduce Theorem \ref{mainthm} to Proposition~\ref{mainproposition} in Section \ref{pfmainthm1}.
Section \ref{preliminary} contains some preliminaries and Section \ref{keylemmas} is devoted to providing several key lemmas which are essential in the proof of Proposition \ref{mainproposition}. 
In the last section, we provide a detailed proof of Proposition \ref{mainproposition}.

\hfill

{\bf Notations.}
Let $\bbn$ and $\bbz$ be the sets of all natural numbers and all integers, respectively. Let $\bbn_0:=\bbn\cup \{0\}$.
We use the symbol $A\lesssim B$ to indicate that $A\leq CB$ for some constant $C>0$ independent of the variable quantities $A$ and $B$, and $A\sim B$ if $A\lesssim B$ and $B\lesssim A$ hold simultaneously.
For simplicity, we adopt the notation $\xxxi:=(\xi_1,\dots,\xi_m)\in (\bbrn)^m $. 
 $\mathscr{C}^L(\bbr)$ consists of functions on $\bbr$ of continuous derivatives up to order $L$. $\mathscr{S}(\bbrn)$ is the class of Schwartz functions on $\rn$.

\hfill

\section{Proof of Theorem \ref{mainthm2}}\label{pfthm2}
Let $0<s<1$.
Suppose that there exists a subset  $J\subseteq \{1,\dots,m\}$ with $|J|\ge 2$ such that
  $$ \sum_{j\in J} \Big(s- \f{1}{p_j} \Big) < -(1-s),$$
which is equivalent to
\begin{equation}\label{equicondition}
\f{1}{q} + \f{|J|}{q'} < \sum_{j\in J} \f{1}{p_j}=:\f{1}{p_J}
\end{equation}
by setting $q:=\frac{1}{1-s}$, where $q'=\tf1s$ is the index conjugate to $q$.
  Here, we notice that $0<p_J<1$ as $|J|\ge 2$ and $\tf1{p_J}\ge 1+\tf{|J|-1}s$, and $1/p_J<|J|$ since $1<p_j$ for all $j\in J$.
Then we will show that there exists an  $\Omega \in L^q( S^{mn-1})$ with mean value zero such that 
$$\Vert \LL_{\Omega}\Vert_{L^{p_1}\times \cdots \times L^{p_m}\to L^p}=\infty.$$

Let $v_n = |B(0,1)|$ be the  volume of the unit ball in ${\mathbb R^n}$. 
For any natural numbers $N$ greater than $2$, we define 
$$f_j^N(y):=\begin{cases}
v_n^{-1/p_j} 2^{Nn/p_j} \chi_{B(0,2^{-N})}(y),&  j\in J\\
v_n^{-1/p_j} 2^{2n/p_j}\chi_{B(0,2^{-2})}(y),& j \in J_m\setminus J
\end{cases}$$
 so that the $L^{p_j} $  norms of $f_j$ are equal to $1$ for all $j\in J_m$. 

For $k\ge 2$, we define
$V_k^J$ to be a tubular neighborhood of  radius comparable to $2^{-k}$ of the subspace $\{(x_1,\dots , x_m)\in (\bbrn)^m: x_{\ii}=x_j \text{ for }\ii,j\in J\}.$  
Precisely, we define
\[
V_k^J :=\bigcup_{x_0\in\bbrn}\Big\{ (y_1,\dots , y_m)\in (\bbrn)^m: \;|x_0-y_j| < \f{4}{3\sqrt{|J|}} 2^{-k} \q\text{ for }~j\in J\Big\} .
\]
 Then we define the function 
\[
\omega_k(\vec{z}\,):= 2^{kn( |J|-1/p_J )  }\chi_{V_k^J\cap S^{mn-1}}(\vec{z}\,) 
\]
on the sphere. 
 We observe that the spherical measure of $V_k^J\cap {S^{mn-1}}$ is proportional to  $2^{-kn (|J|-1) }$ as we have freedom on variables $y_j$ for $j\in J_m\setminus J$ and also on only one variable among $y_j$ for $j\in J$.
 Therefore 
 \[
 \int_{S^{mn-1}} \omega_k \, d\sigma \sim 2^{kn (  (|J| - 1/p_J)   -(|J|-1) )  } 
 =  2^{-kn ( 1/p_J-1 )  } .
 \]
As $p_J<1$, this expression tends to $0$ like $2^{-\epsilon k}$ as $k\to \infty.$  
 We set 
 \[
 \Omega_k:= \omega_k - \alpha_k  \chi_{(V_{2}^J)^c\cap {S^{mn-1}}} ,
 \]
where $\alpha_k  $ is a positive 
constant chosen so that $\Omega_k$ has vanishing integral.
 Note that $\alpha_k\sim  2^{- kn(1/p_J-1)}$ and
 \begin{align}\label{omegakest}
\big\|\Omega_k\big\|_{L^q(\mathbb{S}^{mn-1})}^q &\le 2^{knq(|J|-1/p_J)}\sigma(V_k^J\cap \mathbb{S}^{mn-1})+\alpha_k^q \nonumber\\
&\lesssim 2^{-kn(|J|-1-q(|J|-1/p))}+2^{-knq(1/p_J-1)}\lesssim 2^{-\epsilon' kn}
\end{align}
where $\epsilon':=\min\big\{|J|-1-q(|J|-1/p_J),q(1/p_J-1)\big\}>0$, in view of \eqref{equicondition}, which is equivalent to
$|J|-1-q(|J|-1/p_J)>0$.
We now set  
\[
\Omega:= \sum_{k=2}^\nf k\;\Omega_k 
\]
and then the estimate \eqref{omegakest} clearly yields $\Omega \in L^q(\mathbb{S}^{mn-1})$. 

We now see that for $x\in \mathbb R^n$  satisfying $1<|x|<2$ we have 
\begin{align*}
&\LL_{\Omega}\big(f_1^N,\dots , f_m^N\big)(x) =\sum_{k=2}^{\nf}\, k \LL_{\Omega_k} \big(f_1^N,\dots , f_m^N\big)(x)\nonumber\\
&\qq =\sum_{k=2}^{\nf} k \LL_{\omega_k} \big(f_1^N,\dots , f_m^N\big)(x) -  \sum_{k=2}^{\nf} k\, \alpha_k \,  \LL_{ \chi_{(V_{2}^J)^c\cap {\mathbb{S}^{mn-1}}}  }  \big(f_1^N,\dots , f_m^N\big)(x)\\
&\qq=: \Xi_1(x)-\Xi_2(x)
\end{align*}
The term $\Xi_2$ is an error term for sufficiently large $N$. Indeed,
if $1<|x|<2$ and
 $$|x-y_j|\le\begin{cases}
  2^{-N},& j\in J\\
  2^{-2},& j\in J_m\setminus J
  \end{cases},$$
   then
\begin{equation}\label{rangeofy}
  \frac{3}{4}\sqrt{m}< (|x|-2^{-2})\sqrt{m}   \le |\yyy|\le (|x|+2^{-2})\sqrt{m} < \frac{9}{4}\sqrt{m}
  \end{equation}
and thus, 
\begin{equation*}
  \LL_{ \chi_{(V_{2}^J)^c\cap {\mathbb{S}^{mn-1}}}  }  \big(f_1^N,\dots , f_m^N\big)(x)\le   \LL_{ \chi_{ {\mathbb{S}^{mn-1}}}  }  \big(f_1^N,\dots , f_m^N\big)(x)    \lesssim 2^{-Nn(|J|-1/p_J)},
\end{equation*}
which yields that
\begin{equation}\label{errorest}
 \Xi_2(x)\lesssim 2^{-Nn(|J|-1/p_J)}\sum_{k=2}^{\nf}k \, 2^{-kn(1/p_J-1)}\lesssim 2^{-Nn(|J|-1/p_J)}.
  \end{equation}
Moreover, \eqref{rangeofy} implies that
$$\bigg|\frac{x}{|\yyy|}-\frac{y_j}{|\yyy|}\bigg|\le 2^{-N}|\yyy|^{-1}<\frac{4}{3\sqrt{m}}2^{-N}\le \frac{4}{3\sqrt{|J|}}2^{-N} \q \text{ for all }~j\in J$$
and thus $\yyy/|\yyy|\in V_N^J$.  In other words, $\om_N(\yyy)=2^{kn( |J|-1/p_J )  }$ for $\yyy$ satisfying \eqref{rangeofy}. This combined with the fact that $\LL_{\om_k}\big(f_1^N,\dots,f_m^N\big)\ge0$ shows that
\begin{align*}
\Xi_1(x)&\ge N \LL_{\om_N}\big(f_1^N,\dots,f_m^N\big)(x)\\
&\gtrsim N 2^{Nn(|J|-1/p_J)}\int_{(\bbrn)^m}  f_1^N(y_1)\cdots f_m^N(y_m)   d\yyy\sim N.
\end{align*} 
This, together with \eqref{errorest}, proves that for sufficiently large $N$
$$\LL_{\Omega}\big(f_1^N,\dots , f_m^N\big)(x)\gtrsim N\q\text{when }~1<|x|<2$$
 and thus
$$\big\Vert \LL_{\Omega}(f_1^N,\dots,f_m^N)\big\Vert_{L^p(\bbrn)}\ge \big\Vert \LL_{\Omega}(f_1^N,\dots,f_m^N)\big\Vert_{L^p(\{x\in \bbrn:1<|x|<2\})}\gtrsim N.$$
Since $N$ can be taken arbitrary large,
we conclude the proof.

 \section{Proof of Theorem \ref{mainthm} }\label{pfmainthm1}
We choose a Schwartz function $\Phi^{(m)}$ on $(\bbrn)^m$ such that its Fourier transform $\wh{\Phi^{(m)}}$ is supported in the annulus $\{\xxxi\in (\bbrn)^m: 1/2\le |\xxxi|\le 2\}$ and satisfies
$\sum_{j\in\bbz}\wh{\Phi^{(m)}_j}(\xxxi)=1$ for $\xxxi\not= \000$ where $\wh{\Phi_j^{(m)}}(\xxxi):=\wh{\Phi^{(m)}}(\xxxi/2^j)$.
Recall that $K(\yyy):=\frac{\Omega(\yyy')}{|\yyy|^{mn}}$ for $ \vec y \neq 0$.
For $\ga\in\bbz$  let
 $$ K^{\ga}(\yyy):=\wh{\Phi^{(m)}}(2^{\ga}\yyy)K(\yyy), \quad \yyy\in (\bbrn)^m$$
 and then we observe that $K^\ga(\yyy)=2^{\ga mn} K^0(2^\ga \yyy)$.
 For $\mu\in\bbz$ we define
 \begin{equation}\label{kernelcharacter}
K_{\mu}^{\ga}(\yyy):=\Phi_{{\mu}+\ga}^{(m)}\ast K^{\ga}(\yyy)=2^{\ga mn}[\Phi_{{\mu}}^{(m)}\ast K^{0}](2^\ga \yyy)=2^{\ga mn}K_{\mu}^0(2^{\ga}\yyy)
\end{equation}
and 
$$K_{\mu}(\yyy):=\sum_{\ga\in \bbz}{K_{\mu}^{\ga}(\yyy)}.$$
Then the multilinear operator associated with the kernel $K_{\mu}$ is defined by
 \begin{equation*}
 \LL_{\mu}\big(f_1,\dots,f_m\big)(x):=\int_{(\bbrn)^m}{K_{\mu}(\yyy)\prod_{j=1}^{m}f_j(x-y_j)} ~ d\yyy, \q x\in\bbrn
 \end{equation*} so that
 \begin{align}
 \big\Vert \LL_{\Om}(f_1,\dots,f_m)\big\Vert_{L^{p}(\bbrn)}&\lesssim \Big\Vert \sum_{{\mu}< 0}{\LL_{\mu}(f_1,\dots,f_m)}\Big\Vert_{L^{p}(\bbrn)}\nonumber\\
  &\qq +\Big\Vert \sum_{{\mu}\ge 0}{\LL_{\mu}(f_1,\dots,f_m)}\Big\Vert_{L^{p}(\bbrn)} .\label{secondterm}
 \end{align} 
 
 First of all, using the argument in \cite[Proposition 3]{Gr_He_Ho2018}, we can prove that 
 \begin{equation}\label{e09022}
 \big\Vert \LL_{\mu}(f_1,\dots,f_m)\big\Vert_{L^p(\bbrn)}\lesssim \Vert \Omega\Vert_{L^q(\mathbb{S}^{mn-1})}\Big( \prod_{j=1}^{m}\Vert f_j\Vert_{L^{p_j}(\bbrn)}\Big) \begin{cases}
2^{(mn-\delta){\mu}}, & {\mu}\ge 0\\
2^{(1-\delta){\mu}}, & {\mu}<0
 \end{cases}
 \end{equation}  
 for all $1<q<\infty$ and $0<\delta<1/q'$. 
 This implies that
 \begin{equation*}
 \Big\Vert \sum_{\mu<0}{\LL_{\mu}(f_1,\dots,f_m)}\Big\Vert_{L^{p}(\bbrn)}\lesssim \Vert \Omega\Vert_{L^2(\mathbb{S}^{mn-1})}\prod_{j=1}^{m}\Vert f_j\Vert_{L^{p_j}(\bbrn)}.
 \end{equation*}
 It remains to estimate the term (\ref{secondterm}), but this can be reduced to proving that for $\mu\ge 0$, there exists $\delta_0>0$, possible depending on $p_1,\dots,p_m$, such that
 \begin{equation}\label{mainmaingoal}
 \big\Vert {\LL_{\mu}(f_1,\dots,f_m)}\big\Vert_{L^{p}(\bbrn)}\lesssim 2^{-\delta_0 \mu}\Vert \Om \Vert_{L^\frac{1}{1-s}(\mathbb{S}^{mn-1})}\prod_{j=1}^{m}\Vert f_j\Vert_{L^{p_j}(\bbrn)},
 \end{equation}  which compensate the estimate (\ref{e09022}) for $\mu \ge 0$.
It is already known in \cite[(26)]{Paper1} that 
\begin{equation}\label{222est}
\big\Vert \LL_{\mu}(f_1,\dots,f_m) \big\Vert_{L^{2/m}(\bbrn)} \lesssim_{\wt{\epsilon}}2^{-\wt{\delta} \mu}\Vert \Omega\Vert_{L^2(\mathbb{S}^{mn-1})}\prod_{j=1}^{m}\Vert f_j\Vert_{L^{2}(\bbrn)}
\end{equation}
for some $\wt{\delta}>0$.
In order to achieve the estimate (\ref{mainmaingoal}), we shall use interpolation methods between (\ref{222est}) and the estimates in the following proposition.
\begin{proposition}\label{mainproposition}
Let $1/2\le s<1$, $1/m<p<\infty$, and $1< p_j< \infty$ with $1/p=1/p_1+\cdots+1/p_m$ and $(1/p_1,\dots,1/p_m)\in \HH^m(s)$.
Suppose that $\mu\ge 0$, $\Omega\in L^{\frac{1}{1-s}}(\mathbb{S}^{mn-1})$, and $\int_{\mathbb S^{mn-1}} \Omega \, d\sigma=0$.
Then for any $0<\epsilon<1$, there exists a constant $C_{\epsilon}>0$ such that
\begin{equation}\label{mainpropoest}
\big\Vert \LL_{\mu}(f_1,\dots,f_m)\big\Vert_{L^{p}(\bbrn)}\le C_{\epsilon}2^{\epsilon \mu}\Vert \Om\Vert_{L^{\frac{1}{1-s}}(\mathbb{S}^{mn-1})}\prod_{j=1}^{m}\Vert f_j\Vert_{L^{p_j}(\bbrn)}
\end{equation}
for Schwartz functions $f_1,\dots,f_m$ on $\bbrn$.
\end{proposition}
The proof of the proposition will be provided in the last section.

\medskip

We present a multilinear version of the Marcinkiewicz interpolation theorem, which is a straightforward corollary of \cite[Theorem 1.1]{Gr_Li_Lu_Zh2012} or \cite[Theorem 3]{Ja1988}.
\begin{lemma}\cite{ Gr_Li_Lu_Zh2012,Ja1988}\label{interpollemma}
Let $0<p_{j}^{\ii}\le \infty$ for each $j\in J_m$ and  $\ii=0,1,\dots, m$, and $0<p^{\ii}\le \infty$ satisfy $1/p^{\ii}=1/p^{\ii}_{1}+\dots+1/p^{\ii}_{m}$ for $\ii=0,1,\dots, m$.
Suppose that $T$ is an $m$-linear operator having the mapping properties
\begin{equation*}
\big\Vert T(f_1,\dots,f_m) \big\Vert_{ L^{p^{\ii},\infty}(\bbrn)}\le M_{\ii}\prod_{j=1}^{m}\Vert f_j\Vert_{L^{p_{j}^{\ii}}(\bbrn)}, \quad \ii=0,1,\dots, m
\end{equation*} 
for Schwartz functions $f_1,\dots, f_m$ on $\bbrn$.
Given  $0<\theta_{\ii}<1$ with $\sum_{\ii=0}^m\theta_{\ii}=1$,  set 
\begin{equation*}
\frac 1{ p_j}= \sum_{\ii=0}^m \frac{ \theta_{\ii}}{p_j^\ii}  , \q j\in J_m,\qquad 
\frac 1{p}=\sum_{\ii=0}^m  \frac{ \theta_{\ii}}{p^\ii }   .
\end{equation*} 
Then for Schwartz functions $f_1,\dots, f_m$ on $\bbrn$ we have
\begin{equation*}
\big\Vert T(f_1,\dots,f_m)\big\Vert_{L^{p,\infty}(\bbrn)}\lesssim 
M_0^{ \theta_0}\cdots M_m^{\theta_m} \prod_{j=1}^{m}\Vert f_j\Vert_{L^{p_j}(\bbrn)}.
\end{equation*}
Also, if the  points $(\frac{1}{p_1^\ii},\dots , \frac{1}{p_m^\ii})$, $0\le \ii\le  m$, form a non trivial open simplex in $\mathbb R^m$, then 
\begin{equation*}
\big\Vert T(f_1,\dots,f_m)\big\Vert_{L^{p}(\bbrn)}\lesssim M_0^{ \theta_0}\cdots M_m^{\theta_m}\prod_{j=1}^{m}\Vert f_j\Vert_{L^{p_j}(\bbrn)} .
\end{equation*} 
\end{lemma}

Now taking Proposition~\ref{mainproposition} temporarily for granted, let us complete the proof of (\ref{mainmaingoal}).
We first fix $p_1,\dots,p_m$ such that $P:=(1/p_1,\dots,1/p_m)\in \HH^m(s)$ is not equal to $T:=(1/2,\dots,1/2)$.
Then there exists the unique point $Q:=(1/q_1,\dots,1/q_m)$ on the boundary of $\HH^m(s)$ such that
$$(1-\theta)T+\theta Q=P$$
for some $0<\theta<1$.
Now let $R:=(1/r_1,\dots,1/r_m)$ be the middle point of $P$ and $Q$. We note that $R$ is inside $\HH^m(s)$ because $\HH^m(s)$ is convex.
Since $R=\frac{1}{2}P+\frac{1}{2}Q$, we have
\begin{equation}\label{trp}
(1-\wt{\theta}) T+\wt{\theta} R=P
\end{equation}
where
$$ 0<\wt{\theta}:=\frac{2}{1/\theta+1}<1.$$
Here, $\wt{\theta}$ definitely depends on the point $P$ as $\theta$ does. 
Moreover, since $R$ is contained in the open set $\HH^m(s)$, we may choose $m$ distinct points $R^1,\dots,R^m \in \HH^m(s)$ such that
$R\not= R^{\ii}$ for all $\ii\in J_m$ and
$$R=\theta_1R^1+\dots+\theta_m R^m$$
for some $0<\theta_1,\dots,\theta_m<1$ with $\theta_1+\dots+\theta_m=1$.
This, together with \eqref{trp}, clearly yields that
\begin{equation}\label{trpm}
P=(1-\wt{\theta})T+\wt{\theta}\theta_1R^1+\dots+\wt{\theta}\theta_mR^m
\end{equation}
where $0<1-\wt{\theta}<1$, $0<\wt{\theta}\theta_\ii<1$, and $(1-\wt{\theta})+\wt{\theta}\theta_1+\dots+\wt{\theta}\theta_m=1$.
From the estimate (\ref{222est}), it follows that
\begin{equation}\label{interpol1}
\Vert \LL_{\mu}\Vert_{L^2\times\cdots\times L^2\to L^{2/m}}\lesssim 2^{-\wt{\delta}\mu}\Vert \Omega\Vert_{L^{\frac{1}{1-s}}(\mathbb{S}^{mn-1})}\q \text{ at }~ T=(1/2,\dots,1/2)
\end{equation}
where the embedding $L^{\frac{1}{1-s}}(\mathbb{S}^{mn-1})\hookrightarrow L^{2}(\mathbb{S}^{mn-1})$ is applied.
On the other hand, letting $\epsilon_P:=\frac{(1-\wt{\theta})\wt{\delta}}{2\wt{\theta}}>0$, Proposition \ref{mainproposition} deduces that for each $\ii\in J_m$
\begin{equation}\label{interpol2}
\Vert \LL_{\mu}\Vert_{L^{r_1^{\ii}}\times\cdots\times L^{r_m^{\ii}}\to L^{r^{\ii}}}\lesssim 2^{\epsilon_P\mu}\Vert \Omega\Vert_{L^{\frac{1}{1-s}}(\mathbb{S}^{mn-1})}\q \text{ at }~ R^{\ii}=(1/r_1^{\ii},\dots,1/r_m^{\ii})\in\mathcal{H}(s)
\end{equation}
where $1/r^{\ii}=1/r_1^{\ii}+\dots+1/r_m^{\ii}$.
Now interpolation, stated in Lemma \ref{interpollemma}, between (\ref{interpol1}) and   $m$ points in (\ref{interpol2}) yields 
$$\Vert \LL_{\mu}\Vert_{L^{p_1}\times \cdots\times L^{p_m}\to L^p}\lesssim 2^{-\mu[(1-\wt{\theta})\wt{\delta}-\wt{\theta}\epsilon_P]}\Vert \Omega\Vert_{L^{\frac{1}{1-s}}(\mathbb{S}^{mn-1})},$$
in view of (\ref{trpm}). 
Here, a straightforward computation shows that
$$(1-\wt{\theta})\wt{\delta}-\wt{\theta}\epsilon_P=\frac{(1-\wt{\theta})\wt{\delta}}{2}.$$
See Figure \ref{figinterpolation} for the interpolation.
 \begin{figure}[h]
\begin{tikzpicture}

\path[fill=red!5] (0,0,3)--(1.56,0,3)--(2.25,0,1.95)--(2.25,0,0)--(2.25,1.495,0)--(1.56,2.25,0)--(0,2.25,0)--(0,2.25,1.95)--(0,1.495,3)--(0,0,3);

\draw[dash pattern= { on 2pt off 1pt}](1.56,1.495,3)--(1.56,2.25,1.95)--(2.25,1.495,1.95)--(1.56,1.495,3);
\draw[dash pattern= { on 2pt off 1pt}] (1.56,1.495,3)--(1.56,0,3)--(2.25,0,1.95)--(2.25,1.495,1.95);
\draw[dash pattern= { on 2pt off 1pt}] (1.56,2.25,1.95)--(1.56,2.25,0)--(2.25,1.495,0)--(2.25,1.495,1.95);
\draw[dash pattern= { on 2pt off 1pt}] (1.56,1.495,3)--(0,1.495,3)--(0,2.25,1.95)--(1.56,2.25,1.95);
\draw[dash pattern= { on 2pt off 1pt}] (2.25,0,1.95)--(2.25,0,0)--(2.25,1.495,0);
\draw[dash pattern= { on 2pt off 1pt}] (0,2.25,1.95)--(0,2.25,0)--(1.56,2.25,0);
\draw[dash pattern= { on 2pt off 1pt}] (1.56,0,3)--(0,0,3)--(0,1.495,3);

\draw[dash pattern= { on 1pt off 1pt}] (0,0,0)--(2.25,0,0);
\draw[dash pattern= { on 1pt off 1pt}] (0,0,0)--(0,2.25,0);
\draw[dash pattern= { on 1pt off 1pt}] (0,0,0)--(0,0,3);


\filldraw[fill=black] (1.125,1.125,1.5)  circle[radius=0.3mm];
\draw[-] (1.125,1.125,1.5)--(-1,1.125,0.8);
\node[left] at (-1,1.125,0.8) {\tiny$T=(\frac{1}{2},\frac{1}{2},\frac{1}{2})$};

\filldraw[fill=black] (2.25,0.5,0.3)  circle[radius=0.3mm];
\draw[-] (2.25,0.5,0.3)--(2.4,0.5,-0.8);
\node[right] at (2.4,0.5,-0.8) {\tiny$Q=(\frac{1}{q_1},\frac{1}{q_2},\frac{1}{q_3})\in \partial\mathcal{H}^3(s)$};

\draw[dash pattern= { on 1pt off 2pt}] (1.125,1.125,1.5)--(2.25,0.5,0.3);

\filldraw[fill=black] (2.025,0.625,0.54)  circle[radius=0.5mm];
\draw[-] (2.025,0.625,0.54)--(2.25,2,-1);
\node[above right] at (2.25,2,-1) {\tiny$P=(\frac{1}{p_1},\frac{1}{p_2},\frac{1}{p_3})\in \mathcal{H}^3(s)$};

\filldraw[fill=black] (2.1375,0.5625,0.42)  circle[radius=0.3mm];
\draw[-] (2.1375,0.5625,0.42)--(2.4,1.2,-1);
\node[above right] at (2.4,1.2,-1) {\tiny$R=(\frac{1}{r_1},\frac{1}{r_2},\frac{1}{r_3})\in \mathcal{H}^3(s)$};


\filldraw[fill=black] (1.125+\gap,1.125,1.5)  circle[radius=0.3mm];
\node[left] at (1.125+\gap,1.125,1.5) {\tiny$T$};

\draw[dash pattern= { on 1pt off 2pt}] (1.125+\gap,1.125,1.5)--(2.19375+\gap,0.53125,0.36);

\filldraw[fill=black] (2.025+\gap,0.625,0.54)  circle[radius=0.5mm];
\node[below left] at (2.025+\gap,0.625,0.54) {\tiny$P$};

\filldraw[fill=black] (2.2+\gap,0.2,0)  circle[radius=0.3mm];
\node[right] at (2.16+\gap,0.22,0) {\tiny$R^1\in \HH^3(s)$};

\filldraw[fill=black] (2.2+\gap,0.9,0.36)  circle[radius=0.3mm];
\node[above] at (2.8+\gap,0.85,0.36) {\tiny$R^2\in \HH^3(s)$};

\filldraw[fill=black] (2+\gap,0.2,0.36)  circle[radius=0.3mm];
\node[below] at (2.4+\gap,0.25,0.36) {\tiny$R^3 \in \HH^3(s)$};

\path[fill=blue!5] (2.2+\gap,0.2,0)--(2.2+\gap,0.9,0.36)--(2+\gap,0.2,0.36)--(2.2+\gap,0.2,0);

\draw[dash pattern= { on 1pt off 2pt}] (2.19375+\gap,0.53125,0.36) --(2.2+\gap,0.2,0);

\draw[dash pattern= { on 1pt off 2pt}] (2.19375+\gap,0.53125,0.36) --(2.2+\gap,0.9,0.36);

\draw[dash pattern= { on 1pt off 2pt}] (2.19375+\gap,0.53125,0.36) --(2+\gap,0.2,0.36);

\draw[dotted] (2.2+\gap,0.2,0)--(2.2+\gap,0.9,0.36)--(2+\gap,0.2,0.36)--(2.2+\gap,0.2,0);

\filldraw[fill=black] (2.19375+\gap,0.53125,0.36)  circle[radius=0.3mm];
\node[right] at (2.1+\gap,0.55,0.36) {\tiny$R$};

\end{tikzpicture}
\caption{$(1-\wt{\theta})T+\wt{\theta}R=P$~~ and ~~ $\theta_1R^1+\theta_2R^2+\theta_3R^3=R$~ for~ $m=3$}\label{figinterpolation}
\end{figure}
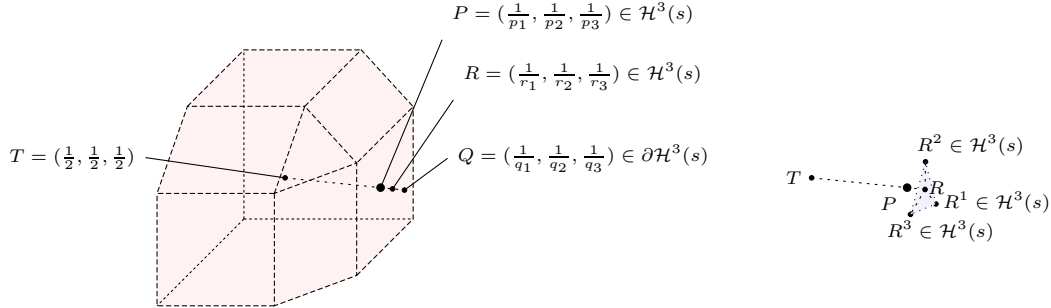

Finally, (\ref{mainmaingoal}) follows from choosing $\delta_0=\frac{(1-\wt{\theta})\wt{\delta}}{2}$ and
this completes the proof of Theorem \ref{mainthm}.

\section{Preliminaries for Proposition \ref{mainproposition}}\label{preliminary}
 Let $\phi$ be a Schwartz function on $\bbrn$ whose Fourier transform is supported in the annulus $\{\xi\in\bbrn: 1/2\le |\xi|\le 2\}$ and satisfies $\sum_{\ga\in\bbz}\wh{\phi_{\ga}}(\xi)=1$ for $\xi\not= 0$, where $\phi_{\ga}:=2^{\ga n}\phi(2^\ga \cdot)$.
For each $\ga \in \bbz$, we define the convolution operator $\La_{\ga}$  by $\La_{\ga}f:=\phi_{\ga}\ast f$.

\subsection{Maximal inequalities}
 Let $\mathcal{M}$  be the Hardy-Littlewood maximal operator, defined by
 \begin{equation*}
 \mathcal{M}f(x):=\sup_{Q: x\in Q}{\frac{1}{|Q|}\int_Q{|f(y)|}dy}
 \end{equation*} where the supremum ranges over all cubes containing $x$, and for $0<t<\infty$ let $\mathcal{M}_tf:=\big(  \mathcal{M}(|f|^t) \big)^{1/t}$.
 Then the maximal operator $\mathcal{M}_t$ is bounded on $L^p(\bbrn)$ for $t<p\le \infty$ and more generally, for $t<p,q<\infty$, we have
\begin{equation}\label{hlmax}
\Big\Vert  \Big(\sum_{\ga\in\bbz}{(\mathcal{M}_{t}f_{\ga})^q}\Big)^{1/{q}} \Big\Vert_{L^p(\bbrn)} \lesssim  \Big\Vert \Big( \sum_{\ga\in\bbz}{|f_{\ga}|^q}  \Big)^{1/{q}}  \Big\Vert_{L^p(\bbrn)}.
\end{equation}   
See \cite[Theorem 5.6.6]{CFA}.
The inequality (\ref{hlmax}) also holds for $0<p\le \infty$ and $q=\infty$.

\subsection{Compactly supported wavelets}
For any fixed $L\in\bbn$ one can construct real-valued compactly supported functions $\psi_F, \psi_M$ in  $\mathscr{C}^L(\bbr)$ satisfying the following properties:
  $\Vert \psi_F\Vert_{L^2(\bbr)}=\Vert \psi_M\Vert_{L^2(\bbr)}=1$, 
  $\int_{\bbr}{x^{\alpha}\psi_M(x)}dx=0$ for all $0\le \alpha \le L$, and moreover, 
if     $\Psi_{\GGG}$ is a function on $\bbr^{mn}$, defined by
$$\Psi_{\GGG}(\xxx):=\psi_{g_1}(x_1)\cdots \psi_{g_{mn}}(x_{mn})$$
for $\xxx:=(x_1,\dots,x_{mn})\in \bbr^{mn}$ and $\GGG:=(g_1,\dots,g_{mn})$ in the set     
 $$\II:=\big\{\GGG:=(g_1,\dots,g_{mn}):g_{\ii}\in\{F,M\} \big\},$$
then the family of functions
\begin{equation*}
\bigcup_{\la\in\bbn_0}\bigcup_{\kkk\in \bbz^{mn}}\big\{ 2^{\la{mn/2}}\Psi_{\GGG}(2^{\la}\xxx-\kkk):\GGG\in \II^{\la}\big\}
\end{equation*}
forms an orthonormal basis of $L^2(\bbr^{mn})$,
where $\II^0:=\II$ and  for $\la \ge 1$, we set  $\II^{\la}:=\II\setminus \{(F,\dots,F)\}$.

It is known in \cite[Theorem 1.64]{Tr2006} that
if $L$ is sufficiently large, then every  $H\in L^q(\bbr^{mn})$ with  $1<q<\infty$   can be represented as
\begin{equation}\label{daubechewavelet}
H(\xxx)=\sum_{\la\in\bbn_0}\sum_{\GGG\in\II^{\la}}\sum_{\kkk\in \bbz^{mn}}b_{\GGG,\kkk}^{\la}2^{\la mn/2}\Psi_{\GGG}(2^{\la} \xxx -\kkk)
\end{equation}
with the right hand side convergence in $\mathscr S'(\bbr^{mn})$,
and 
\begin{equation}\label{lqlqs}
\Big\Vert \Big(\sum_{\vec G\in\II^{\la}}\sum_{ \vec k\in\bbz^{mn}} \big|b_{\vec G,\vec k}^\la \Psi^\la_{\vec G,\vec k}\big|^2\Big)^{1/2} \Big\Vert_{L^q(\bbr^{mn})} \lesssim \|H\|_{L^q(\bbr^{mn})} 
\end{equation}
where $ \Psi^\la_{\vec G,\vec k}(\vec x)=2^{\la mn/2}\Psi_{\GGG}(2^{\la} \xxx -\kkk)$,   
\begin{equation*}
b_{\GGG,\kkk}^{\la}:=\int_{\bbr^{mn}}{H(\xxx)\Psi^\la_{\vec G,\vec k}(\vec x)}d\xxx .
\end{equation*}
Moreover, it follows from (\ref{lqlqs}) and the disjoint support property  of the $\Psi^\la_{\vec G,\vec k}$'s that 
\begin{align}
\big\Vert \big\{b_{\GGG,\kkk}^{\la}\big\}_{\kkk\in \bbz^{mn}}\big\Vert_{\ell^{q}}
\approx&\Big(2^{\la mn(1-q/2)}\int_{\bbr^{mn}}\Big( \sum_{\vec k} \big|b^\la_{\vec G,\vec k}\Psi^\la_{\vec G,\vec k}(\vec x)\big|^2\Big)^{q/2}d\vec x\Big)^{1/q}\notag \\
\lesssim& \; 2^{-\la mn (1/2-1/q)}\Vert H\Vert_{L^q(\bbr^{mn})}. \label{lqestimate}
\end{align} 

Throughout, we will consistently use the notation  $G_j:=(g_{(j-1)n+1},\dots,g_{jn})$ for an element of $ \{F,M\}^n$  and  
$\Psi_{G_j}(\xi_j):=\psi_{g_{(j-1)n+1}}(\xi_{j}^1)\cdots \psi_{g_{jn}}(\xi_j^n)$ for $\xi_j:=(\xi_j^1,\dots,\xi_j^n)\in \bbrn$
so that $\GGG=(G_1,\dots,G_m)\in (\{F,M\}^n)^m$ and $
\Psi_{\GGG}(\xxxi)=\Psi_{G_1}(\xi_1)\cdots \Psi_{G_m}(\xi_m).
$
For each $\kkk:=(k_1,\dots,k_m)\in (\bbzn)^m$ and $\la\in \bbn_0$, let 
$$
\Psi_{G_j,k_j}^{\la}(\xi_j):=2^{\la n/2}\Psi_{G_j}(2^{\la}\xi_j-k_j), \qq 1\le j\le m
$$
and
$$
\Psi_{\GGG,\kkk}^{\la}(\xxxi\,):=\Psi_{G_1,k_1}^{\la}(\xi_1)\cdots \Psi_{G_m,k_m}^{\la}(\xi_m).
$$
We also assume that the  support of 
 $\psi_{g_{j}}$ is contained in $\{\xi\in \bbr: |\xi|\le C_0 \}$ for some $C_0>1$,
which implies that
\begin{equation}\label{supppsi}
\supp (\Psi_{G_j,k_j}^\la)\subset \big\{\xi_j\in\bbrn: |2^{\la}\xi_j-k_j|\le C_0\sqrt{n}\big\}.
\end{equation}
In other words, the support of $\Psi_{G_j,k_j}^\la$ is contained in the ball centered at $2^{-\la}k_j$ and radius $C_0\sqrt{n}2^{-\la}$.


\subsection{Columns and Projections}
We now introduce a few notions and   related combinatorial properties. 
For a  fixed $\kkk\in (\bbzn)^m$, $l\in J_m=\{1,2,\dots,m\}$, and $1\leq j_1<  \dots< j_l\le m$ let 
$$\kkk^{j_1,\dots,j_l}:=(k_{j_1},\dots,k_{j_l})$$ denote the vector in $(\bbzn)^l$ consisting of the $j_1,\dots,j_l$ components of $\kkk$ and 
$\kkk^{*j_1,j_2,\dots,j_l}$ stand for the vector in $(\bbzn)^{m-l}$, consisting of $\kkk$ except for
 the $j_1$, \dots, $j_l$ components
(e.g.  $\kkk^{*1,\dots,j}=\kkk^{j+1,\dots,m}=(k_{j+1},\dots,k_m)\in (\bbzn)^{m-j}$).
For any sets $\UU$ in $(\bbzn)^m$, $j\in J_m$, and $1\le j_1<\dots< j_l \leq m$ let 
$$\mathcal{P}_j\UU:=\big\{k_j\in\bbzn:\kkk \in\UU ~\text{ for some }\kkk^{*j}\in(\bbzn)^{m-1} \big\}$$
$$\mathcal{P}_{*j_1,\dots,j_l}\UU:=\big\{\kkk^{*j_1,\dots,j_l}\in(\bbzn)^{m-l}:\kkk\in \UU ~\text{ for some }k_{j_1},\dots, k_{j_l}\in\bbzn \big\}$$ 
be the projections of $\UU$ onto the $k_j$-column and $\kkk^{*j_1,\dots,j_l}$-plane, respectively.
For a fixed $\kkk^{*j_1,\dots,j_l}\in \mathcal{P}_{*j_1,\dots,j_l}\UU$,  we define 
\begin{equation*}
Col_{\kkk^{*j_1,\dots,j_l}}^{\UU}:=\{\kkk^{j_1,\dots,j_l}\in (\bbzn)^l: \kkk=(k_1,\dots,k_m) \in\UU\}.
\end{equation*}
Then we observe that 
\begin{equation}\label{disjointdecom1}
\sum_{\kkk\in \UU}\cdots = \sum_{\kkk^{*j_1,\dots,j_l}\in \mathcal{P}_{*j_1,\dots,j_l}\UU}\Big( \sum_{\kkk^{j_1,\dots,j_l}\in Col_{\kkk^{*j_1,\dots,j_l}}^{\UU}}\cdots\Big).
\end{equation}
For more details of these notations and their applications, we refer to \cite{Paper1}, while similar ideas go back to \cite{Gr_He_Ho2018}.

 \section{Key Lemmas for the proof of Proposition \ref{mainproposition}} \label{keylemmas}
 
Let $C_0$ be the constant that appeared in \eqref{supppsi}.
 For $\la \in\bbn_0$ satisfying $C_0\sqrt{n}\le 2^{\la+1}$, let
 $$\WW^{\la}:=\big\{ k\in\bbzn: 2C_0\sqrt{n}\le |k|\le 2^{\la+2}     \big\}.$$
 For $\la\in \bbn_0$, $G\in \{F,M\}^n$, $k\in\bbzn$, and $\la\in\bbz$,
we define the operator $L_{G,k}^{\la,\ga}$ via the Fourier transform by
\begin{equation}\label{lgklg}
\big( L_{G,k}^{\la,\ga}f\big)^{\wedge}(\xi):= \Psi_{G,k}^{\la}(\xi/2^{\ga})\wh{f}(\xi), \qq \ga\in\bbz.
\end{equation} 
Then we observe that 
\begin{equation}\label{lmaximalbound}
\big|L_{G,k}^{\la,\ga}f(x)\big|\lesssim 2^{\la n/2}\mathcal{M}f(x) \q \text{ uniformly in the parameters}~ \la,G,k,\ga
\end{equation}
and
 for  $k\in \WW^{\la+\mu}$ with $C_0\sqrt{n}\le 2^{\la+\mu+1}$,
\begin{equation}\label{lgkest}
L_{G,k}^{\la,\ga}f=L_{G,k}^{\la,\ga}f^{\la,\ga,{\mu}}   
\end{equation}
due to the support of $\Psi_{G}$, 
where 
\begin{equation*}
{f^{\la,\ga,\mu}}:=\sum_{j=-\la+c_0}^{\mu+3}\La_{\ga+j}f
\end{equation*} for some $c_0\in \bbn$, depending on $C_0$ and $n$.
It is easy to check that for $1<p<\infty$
\begin{equation*}
\Big\Vert \Big(\sum_{\ga\in\bbz}\big|f^{\la,\ga,\mu} \big|^2 \Big)^{1/2}\Big\Vert_{L^p(\bbrn)}\leq \sum_{j=-\la+c_0}^{\mu+3}\Big\Vert \Big(\sum_{\ga\in\bbz}\big|\La_{\ga+j} f \big|^2 \Big)^{1/2}\Big\Vert_{L^p(\bbrn)}\lesssim (\mu+\la+4)\Vert f\Vert_{L^p(\bbrn)},
\end{equation*}
where the triangle inequality and the Littlewood-Paley theory are applied in the inequalities.
 
 \begin{lemma}\label{variantlp}
 Let $2\le p< \infty$, $1<t<2$,  $u\in\bbzn$, and $s\ge 0$.
 Then we have
 \begin{equation}\label{variantlpest}
 \Big\Vert \Big(\sum_{\ga\in\bbz}\big\Vert\La_{\ga} f(x-2^{s-\ga}\cdot) \;  \Psi_G^{\vee}\big\Vert_{L^t(u+[0,1)^n)}^2 \Big)^{1/2}\Big\Vert_{L^p(x)}\lesssim_M \frac{1}{(1+|u|)^M} \Vert f\Vert_{L^p(\bbrn)}
 \end{equation}
 uniformly in $s\ge 0$.
 \end{lemma}
 \begin{proof}
 Using the fact that
 $$\sup_{y\in u+[0,1)^n}|\Psi_G^{\vee}(y)|\lesssim_{M,t} \frac{1}{(1+|u|)^{M+n/t}} \q \text{ for any }~M>0,$$
 we see that
 \begin{align*}
\big\Vert\La_{\ga} f(x-2^{s-\ga}\cdot) \;  \Psi_G^{\vee}\big\Vert_{L^t(u+[0,1)^n)}&=\Big( \int_{u+[0,1)^n}  \big| \La_{\ga} f(x-2^{s-\ga}y)\big|^t |\Psi_G^{\vee}(y)|^t   dy\Big)^{1/t}\\
&\lesssim \frac{1}{(1+|u|)^{M+n/t}}\Big( \int_{u+[0,1)^n}  \big| \La_{\ga} f(x-2^{s-\ga}y)\big|^t    dy\Big)^{1/t}.
 \end{align*}
  Moreover, using a change of variables, we have
 \begin{align*}
 \Big(\int_{u+[0,1)^n}{\big|\La_{\ga} f(x-2^{s-\ga}y)\big|^{t}}dy \Big)^{1/t}&\lesssim \Big(\frac{1}{2^{(s-\ga)n}}\int_{|y|\le \sqrt{n}(1+|u|)2^{s-\ga}}{\big| \La_{\ga} f(x-y)\big|^{t}}dy \Big)^{1/t}\\
 &\lesssim (1+|u|)^{n/t}\mathcal{M}_{t}\La_{\ga} f(x),
 \end{align*}
 which proves
 \begin{equation*}
 \big\Vert\La_{\ga} f(x-2^{s-\ga}\cdot) \;  \Psi_G^{\vee}\big\Vert_{L^t(u+[0,1)^n)}\lesssim_{M,t}\frac{1}{(1+|u|)^{M}}\mathcal{M}_t\La_{\ga}f(x).
 \end{equation*}
Now the left-hand side of (\ref{variantlpest}) is less than a constant times
\begin{align*}
\frac{1}{(1+|u|)^{M}}\big\Vert \big\{ \mathcal{M}_t\La_{\ga}f\big\}_{\ga\in\bbz}\big\Vert_{L^p(\ell^2)}\lesssim \frac{1}{(1+|u|)^{M}}\big\Vert \big\{ \La_{\ga}f\big\}_{\ga\in\bbz}\big\Vert_{L^p(\ell^2)}\sim \frac{1}{(1+|u|)^{M}}\Vert f\Vert_{L^p(\bbrn)}
\end{align*}
 by using the maximal inequality (\ref{hlmax}) and the Littlewood-Paley theory.
 \end{proof}

 \begin{lemma}\label{keylemma3}
 Let $2\le p< \infty$, $0<\epsilon<1$ and $\la,\mu\in\bbz$ with $2^{\la+\mu+1}\ge C_0\sqrt{n}$. Suppose that $E^{\la+\mu}$ is a subset of  $\WW^{\la+\mu}$.
 Let $\{b_k^{\ga}\}_{k\in\bbzn}$ be a sequence of complex numbers and
 \begin{equation*}
 \BB_2:= \sup_{\ga\in\bbz}\big\Vert \{b_k^{\ga}\}_{k\in\bbzn}\big\Vert_{\ell^2} \quad \text{ and }\quad \BB_{\infty}:= \sup_{\ga\in\bbz}\big\Vert \{b_k^{\ga}\}_{k\in\bbzn}\big\Vert_{\ell^{\infty}}.
 \end{equation*}
 Then there exists $C_{\epsilon}>0$ such that
  \begin{equation}\label{mainest3}
 \Big\Vert  \Big(\sum_{\ga\in\bbz}{\Big| \sum_{k\in E^{\la+\mu}}b_k^{\ga} L_{G,k}^{\la,\ga}f\Big|^2} \Big)^{1/2}\Big\Vert_{L^{p}(\bbrn)}\le C_{\epsilon} 2^{\la n/2} (\la+\mu+4)\BB_2^{1-\epsilon} \BB_{\infty}^{\epsilon}|E^{\la+\mu}|^{\epsilon}\Vert f\Vert_{L^p(\bbrn)}
 \end{equation}
 for $f\in\mathscr{S}(\bbrn)$.
 \end{lemma}
 \begin{proof}
 Using (\ref{lgkest}), the left-hand side of (\ref{mainest3}) is less than
\begin{equation*}
\sum_{j=-\la+c_0}^{\mu+3}\Big\Vert \Big(\sum_{\ga\in\bbz} \Big| \sum_{k\in E^{\la+\mu}}b_k^{\ga}L_{G,k}^{\la,\ga}\La_{\ga+j}f \Big|^2\Big)^{1/2}\Big\Vert_{L^{p}(\bbrn)}.
\end{equation*} 
Let $t:=\frac{2}{1+\epsilon}$ so that $1<t<2<t'=\frac{2}{1-\epsilon}$.
 Then we apply H\"older's inequality to obtain
  \begin{align*}
& \Big| \sum_{k\in E^{\la+\mu}} b^{\ga}_k L_{G,k}^{\la,\ga}\La_{\ga+j}f(x)\Big| \\
&\le 2^{\la n/2}\int_{\bbrn}{\big| B^{\ga}_{E^{\la+\mu}}(y)\big|\big|\La_{\ga+j}f(x-2^{\la-\ga}y) \big|  \big| \Psi^{\vee}_G(y)\big|}dy\\
&= 2^{\la n/2}\sum_{u\in\bbzn}\int_{u+[0,1)^n}\big| B^{\ga}_{E^{\la+\mu}}(y)\big|\big|\La_{\ga+j}f(x-2^{\la-\ga}y) \big|  \big| \Psi^{\vee}_G(y)\big| dy\\
&\le 2^{\la n/2}\sum_{u\in\bbzn}\big\Vert B_{E^{\la+\mu}}^{\ga}\big\Vert_{L^{t'}(u+[0,1)^n)}\big\Vert \La_{\ga+j}f(x-2^{\la-\ga}\cdot) \Psi_G^{\vee}\big\Vert_{L^t(u+[0,1)^n)}
 \end{align*}
 where
  \begin{equation*}
  B^{\ga}_{E^{\la+\mu}}(x):=\sum_{k\in E^{\la+\mu}}b^{\ga}_k e^{2\pi i\langle x,k\rangle}.
  \end{equation*}
 We first observe that
 \begin{align}\label{bega}
 \big\Vert B^{\ga}_{E^{\la+\mu}}\big\Vert_{L^{t'}(u+[0,1)^n)}&=\big\Vert B^{\ga}_{E^{\la+\mu}}\big\Vert_{L^{t'}([0,1)^n)}\nonumber\\
 &\le \big\Vert B^{\ga}_{E^{\la+\mu}}\big\Vert_{L^2([0,1)^n)}^{2/t'}\big\Vert B^{\ga}_{E^{\la+\mu}}\big\Vert_{L^{\infty}([0,1]^n)}^{1-2/t'} \le \BB_2^{1-\epsilon}\BB_{\infty}^{\epsilon}|E^{\la+\mu}|^{\epsilon}.
 \end{align}
 Therefore, the left-hand side of (\ref{mainest3}) is dominated by a constant times
 \begin{align*}
& 2^{\la n/2}\BB_2^{1-\epsilon}\BB_{\infty}^{\epsilon}|E^{\la+\mu}|^{\epsilon}\sum_{j=-\la+c_0}^{\mu+3}\Big\Vert \Big(\sum_{\ga\in\bbz}\Big( \sum_{u\in\bbzn}\big\Vert \La_{\ga+j}f(x-2^{\la-\ga}\cdot) \Psi_G^{\vee}\big\Vert_{L^t(u+[0,1)^n)}     \Big)^2 \Big)^{1/2}     \Big\Vert_{L^p(\bbrn)}\\
&\le 2^{\la n/2}\BB_2^{1-\epsilon}\BB_{\infty}^{\epsilon}|E^{\la+\mu}|^{\epsilon}\sum_{u\in\bbzn}\sum_{j=-\la+c_0}^{\mu+3}\Big\Vert \Big(\sum_{\ga\in\bbz}\big\Vert \La_{\ga}f(x-2^{\la+j-\ga}\cdot) \Psi_G^{\vee}\big\Vert_{L^t(u+[0,1)^n)}^2     \Big)^{1/2}     \Big\Vert_{L^p(\bbrn)}.
 \end{align*}
 Now it follows from Lemma \ref{variantlp} that the preceding expression is controlled by a constant multiple of
 \begin{align*}
 2^{\la n/2}\BB_2^{1-\epsilon}\BB_{\infty}^{\epsilon}|E^{\la+\mu}|^{\epsilon}(\la+\mu+4)\Vert f\Vert_{L^p(\bbrn)}\sum_{u\in\bbzn}\frac{1}{(1+|u|)^M}
 \end{align*} for $M>n$. The sum over $u\in\bbzn$ is obviously finite and this completes the proof of Lemma \ref{keylemma3}.
 \end{proof}

     \begin{lemma}\label{keylemma32}
     Let $2\le l\le m$, $2\le p_1,\dots,p_l< \infty$, and $0<p< \infty$ with $1/p_1+\dots+1/p_l=1/p$.
 Let  $0<\epsilon<1$ and $\la,\mu\in\bbz$ with $2^{\la+\mu+1}\ge C_0\sqrt{n}$. Suppose that  $E_l^{\la+\mu}$ is a subset of $(\WW^{\la+\mu})^l$.
 Let $\{b_{\kkk}^{\ga}\}_{\kkk\in(\bbzn)^l}$ be a sequence of complex numbers and
 $$\DD_2:=\sup_{\ga\in\bbz}\big\Vert \{b_{\kkk}^{\ga}\}_{\kkk\in (\bbzn)^l}\big\Vert_{\ell^2}\q \text{and}\q \DD_{\infty}:=\sup_{\ga\in\bbz}\big\Vert \{b_{\kkk}^{\ga}\}_{\kkk\in (\bbzn)^{l}}\big\Vert_{\ell^{\infty}}.$$
 Then there exists $C_{\epsilon}>0$ such that
  \begin{align}\begin{split} \label{mainest32}
 \bigg\Vert  \sum_{\ga\in\bbz}& {\Big| \sum_{\kkk\in E_l^{\la+\mu}}  
 b_{\kkk}^{\ga}   \prod_{j=1}^{l}L_{G_j,k_j}^{\la,\ga}f_j  \Big|} \bigg\Vert_{L^{p}(\bbrn)} \\
& \le C_{\epsilon} 2^{\la l n /2} (\la+\mu+4)^{l/\min{\{1,p\}}}   \DD_2^{1-\epsilon} \DD_{\infty}^{\epsilon}|E^{\la+\mu}|^{\epsilon}\prod_{j=1}^{l}\Vert f_j\Vert_{L^{p_j}(\bbrn)}
 \end{split} 
 \end{align}
 for $f_1,\dots,f_l\in\mathscr{S}(\bbrn)$.
 \end{lemma}
\begin{proof}
Using (\ref{lgkest}), the left-hand side of (\ref{mainest32}) is less than
\begin{equation*}
\bigg( \sum_{\ii_1=-\la+c_0}^{\mu+3}\cdots\sum_{\ii_l=-\la+c_0}^{\mu+3}\bigg\Vert \sum_{\ga\in\bbz} \Big| \sum_{\kkk\in E_l^{\la+\mu}}b_{\kkk}^{\ga}\; L_{G_1,k_1}^{\la,\ga}\La_{\ga+\ii_1}f_1\cdots    L_{G_l,k_l}^{\la,\ga}\La_{\ga+\ii_l}f_{l} \Big| \bigg\Vert_{L^{p}(\bbrn)}^{\min{\{1,p\}}}\bigg)^{1/\min{\{1,p\}}}.
\end{equation*}
Choose $t:=\frac{2}{1+\epsilon}$ so that $1<t<2<t'=\frac{2}{1-\epsilon}$ as in the proof of Lemma \ref{keylemma3}.
 Then it follows from H\"older's inequality that
  \begin{align*}
& \Big| \sum_{\kkk\in E_l^{\la+\mu}}b_{\kkk}^{\ga}\; L_{G_1,k_1}^{\la,\ga}\La_{\ga+\ii_1}f_1(x)\cdots    L_{G_l,k_l}^{\la,\ga}\La_{\ga+\ii_l}f_{l}(x) \Big| \\
&\le 2^{\la l n/2}\int_{(\bbrn)^l}{\big| B^{\ga}_{E_l^{\la+\mu}}(\yyy)\big|  \prod_{j=1}^{l}\big|\La_{\ga+\ii_j}f(x-2^{\la-\ga}y_j) \Psi^{\vee}_{G_j}(y_j) \big|}d\yyy\\
&= 2^{\la l n/2}\sum_{\uuu\in (\bbzn)^l}\int_{\uuu+[0,1)^{nl}}{\big| B^{\ga}_{E_l^{\la+\mu}}(\yyy)\big|  \prod_{j=1}^{l}\big|\La_{\ga+\ii_j}f(x-2^{\la-\ga}y_j) \Psi^{\vee}_{G_j}(y_j) \big|}d\yyy\\
&\le 2^{\la l n/2}\sum_{\uuu\in (\bbzn)^l}\big\Vert B_{E_l^{\la+\mu}}^{\ga}\big\Vert_{L^{t'}(\uuu+[0,1)^{nl})} \prod_{j=1}^{l}\big\Vert \La_{\ga+\ii_j}f(x-2^{\la-\ga}\cdot) \Psi_{G_j}^{\vee}\big\Vert_{L^t(u_j+[0,1)^n)}
 \end{align*}
 where $\yyy:=(y_1,\dots,y_l)\in (\bbrn)^l$, $\uuu:=(u_1,\dots,u_l)\in (\bbzn)^l$, and
  \begin{equation*}
  B^{\ga}_{E_l^{\la+\mu}}(\yyy):=\sum_{\kkk\in E_l^{\la+\mu}}b^{\ga}_{\kkk} e^{2\pi i\langle \yyy,\kkk\rangle}.
  \end{equation*}

Similar to (\ref{bega}), we have
\begin{equation*}
\big\Vert B_{E_l^{\la+\mu}}^{\ga}\big\Vert_{L^{t'}(\uuu+[0,1)^{nl})}\lesssim \DD_2^{1-\epsilon} \DD_{\infty}^{\epsilon} |E_l^{\la+\mu}|^{\epsilon}.
\end{equation*}
Thus, the left-hand side of (\ref{mainest32}) is controlled by a constant times
\begin{align*}
&2^{\la l n/2}\DD_{2}^{1-\epsilon}\DD_{\infty}^{\epsilon}\big|E_l^{\la+\mu}\big|  \bigg( \sum_{\uuu\in (\bbzn)^l}\sum_{\ii_1=-\la+c_0}^{\mu+3}\cdots\sum_{\ii_l=-\la+c_0}^{\mu+3}   \\
&\qq \qq\Big\Vert \sum_{\ga\in\bbz}\prod_{j=1}^{l} \big\Vert \La_{\ga+\ii_j}f(x-2^{\la-\ga}\cdot) \Psi_{G_j}^{\vee}\big\Vert_{L^t(u_j+[0,1)^n)}          \Big\Vert_{L^p(x)}^{\min{\{1,p\}}} \bigg)^{1/\min{\{1,p\}}}
\end{align*}
and the $L^p$ norm is less than
\begin{align*}
&\bigg\Vert \prod_{j=1}^{l}\Big( \sum_{\ga\in\bbz} \big\Vert \La_{\ga}f(x-2^{\la+\ii_j-\ga}\cdot) \Psi_{G_j}^{\vee}\big\Vert_{L^t(u_j+[0,1)^n)}^2     \Big)^{1/2}        \bigg\Vert_{L^p(x)}\\
&\le \prod_{j=1}^{l}\Big\Vert\Big( \sum_{\ga\in\bbz} \big\Vert \La_{\ga}f(x-2^{\la+\ii_j-\ga}\cdot) \Psi_{G_j}^{\vee}\big\Vert_{L^t(u_j+[0,1)^n)}^2     \Big)^{1/2}        \Big\Vert_{L^{p_j}(x)}\\
&\lesssim \prod_{j=1}^{l}\frac{1}{(1+|u_j|)^M} \Vert f_j\Vert_{L^{p_j}(\bbrn)}
\end{align*}
for $M>n$, where the Cauchy-Schwarz inequality, H\"older's inequality, and Lemma \ref{variantlp} are applied in the inequalities.
This concludes that the left-hand side of (\ref{mainest32}) is dominated by
\begin{align*}
& 2^{\la l n/2} (\la+\mu+4)^{l/\min{\{1,p\}}}\DD_{2}^{1-\epsilon}\DD_{\infty}^{\epsilon}\big|E_l^{\la+\mu}\big| \Big(\prod_{j=1}^{l}\Vert f_j\Vert_{L^{p_j}(\bbrn)}\Big) \sum_{\uuu\in(\bbzn)^l} \prod_{j=1}^{l}\frac{1}{(1+|u_j|)^M}\\
&\lesssim  2^{\la l n/2} (\la+\mu+4)^{l/\min{\{1,p\}}}\DD_{2}^{1-\epsilon}\DD_{\infty}^{\epsilon}\big|E_l^{\la+\mu}\big| \prod_{j=1}^{l}\Vert f_j\Vert_{L^{p_j}(\bbrn)},
\end{align*}
which completes the proof.
\end{proof}

\begin{lemma}\label{convexhull}
Let $0<s<1$.
For $l\in J_m$ we define
$$\mathscr{V}^m_l(s):=\{(t_1,\dots,t_m): 0< t_l< 1 \q \text{and }~ 0<  t_j< s ~\text{ for }~ j\not= l\}$$
and let 
$\mathbb{H}^m(s)$ be the convex hull of $\mathscr{V}_1^m(s),\dots,\mathscr{V}_m^m(s)$.
Then we have 
\begin{equation}\label{e012}
\HH^m(s)=\mathbb{H}^m(s).
\end{equation}
\end{lemma}
\begin{proof}
Clearly, $\HH^m(s)$ is open and convex as each $\HH_J(s)$, $J\subseteq J_m$, is an open convex set.
It is easy to see that $\mathscr{V}^m_l(s)\subset \HH^m(s)$ for all $l\in J_m$ and thus 
we have
$$\HH^m(s)  \supseteq   \mathbb{H}^m(s)\q \text{ for all }~m\ge 2.$$

Now let's us prove the opposite direction by using induction on the degree $m$ of multilinearity.
We first note that
\begin{equation}\label{e011}
\HH^m(s)=\Big\{ (t_1,\dots,t_m)\in (0,1)^m: \sum_{j=1}^{m}\min\{s-t_j,0\}> s-1\Big\}.
\end{equation}
To verify \eqref{e011}, we denote by $H^m(s)$ the right-hand side of \eqref{e011}.
Suppose that $t=(t_1,\dots,t_m)\in \HH^m(s)$ and  let $J^{(t)}:=\{j\in J_m:\ t_j>s\}$. Then by the definition of $\HH^m(s)$ in \eqref{defhh}, we have
$$
s-1<\sum_{j\in J^{(t)}}(s-t_j)=\sum_{j=1}^{m}\min\{s-t_j,0\},
$$
which implies $\HH^m(s)\subset H^m(s)$.
 Moreover, if $t=(t_1,\dots, t_m)\in H^m(s)$, then
$$
s-1<\sum_{j\in J^{(t)}} (s-t_j)\le \sum_{j\in J\cap J^{(t)}}(s-t_j)\le  \sum_{j\in J\cap J^{(t)}}(s-t_j)+ \sum_{j\in J\setminus J^{(t)}}(s-t_j)= \sum_{j\in J}(s-t_j)
$$
 for any $J\subset J_m.$
This gives that $H^m(s)\subset \HH^m(s) \big( =\cap_J \HH_J^m(s)\big)$.

We now return to the proof of $\HH^m(s)  \subseteq   \mathbb{H}^m(s)$ for $m\ge 2$.
The case $m=2$ is obvious from a simple geometric observation, but we provide an explicit approach.
If $(t_1,t_2)\in \HH^2(s)$, then we have $0< t_1,t_2< 1$ and $0< t_1+t_2< 1+s$.
When either $t_1$ or $t_2$ is less than $s$, then $(t_1,t_2)$ belongs to one of $\mathscr{V}_1^m(s)$ or $\mathscr{V}_2^m(s)$ by definition.
When $s\le t_1,t_2<1$, we choose $0<\epsilon<1$ such that
$$0<\epsilon<1+s-(t_1+t_2).$$
Then the point $(t_1, t_2)$ lies on the segment joining $(t_1+t_2-s+\epsilon,s-\epsilon)\in\mathscr{V}_1^2(s)$ and $(s-\epsilon,t_1+t_2-s+\epsilon)\in\mathscr{V}_2^2(s)$ as $0<s-\epsilon<s$ and $0<t_1+t_2-s+\epsilon<1$.
This shows
\begin{equation*}
\HH^2(s)\subseteq \mathbb{H}^2(s).
\end{equation*} 
Now suppose that $\HH^m(s)\subseteq \mathbb{H}^m(s)$ is true for some $m\ge 2$, and let $(t_1,\dots,t_{m+1})\in \HH^{m+1}(s)$.
If $0<t_{m+1}<s$, then
$$\sum_{j=1}^{m}\min\{s-t_j,0\}=\sum_{j=1}^{m+1}\min\{s-t_j,0\}>s-1$$
so that $(t_1,\dots,t_m)\in \HH^m(s)\subseteq \mathbb{H}^m(s)$ by applying the induction hypothesis.
Therefore, the point $(t_1,\dots,t_m,t_{m+1})$ belongs to the convex hull of the following $m$ sets:
$$\mathscr{V}^m_l(s)\times (0,s)= \mathscr{V}^{m+1}_l(s),\qq l\in J_m.$$
This implies $(t_1,\dots,t_{m+1})\in \mathbb{H}^{m+1}(s)$.
Similarly, the same conclusion also holds if $0<t_j<s$ for some $j\in J_m$.
For the remaining cases, we assume that $s\le t_1,\dots,t_{m+1}<1$.
Since $(t_1,\dots,t_{m+1})\in \HH^{m+1}(s)$, we see that
$t_1+\dots+t_{m+1}<1+ms$, and thus there exists $0<\epsilon<1$ so that
\begin{equation}\label{epsilonrange}
0<m\epsilon<1+ms-(t_1+\dots+t_{m+1}).
\end{equation}
Then the point $(t_1,\dots,t_{m+1})$ is clearly located on the convex hull of the points
\begin{align*}
&(t_1+\dots+t_{m+1}-ms+m\epsilon,s-\epsilon,\dots,s-\epsilon)\in\mathscr{V}_1^{m+1}(s)\\
&\qq\qq\qq\qq\qq\vdots\\
&(s-\epsilon,\dots,s-\epsilon,t_1+\dots+t_{m+1}-ms+m\epsilon)\in\mathscr{V}_{m+1}^{m+1}(s)
\end{align*}
as $0<t_1+\dots+t_{m+1}-ms+m\epsilon<1$, because of \eqref{epsilonrange}.
This proves that $(t_1,\dots,t_{m+1})\in \mathbb{H}^{m+1}(s)$ and completes that proof of $\HH^{m+1}(s)\subseteq \mathbb{H}^{m+1}(s)$.

By induction, we finally have 
$$\HH^m(s)\subseteq \mathbb{H}^{m}(s) \q\text{ for general }~m\ge 2.$$
\end{proof}

\section{Proof of Proposition \ref{mainproposition} }\label{pfmainpropo}
It suffices to prove \eqref{mainpropoest} for $\mu$ such that $2^{\mu-10}>C_0\sqrt{mn}$ in view of \eqref{e09022}.
The proof will be based on mathematical induction starting with the estimate in the following proposition.

\begin{proposition}\label{lppropo}
Let $2\le p_1,\dots,p_m\le \infty$ and $2/m\le p<\infty$ with $1/p_1+\cdots+1/p_m=1/p$.
Suppose that $0<\epsilon<1$ and $2^{\mu-10}>C_0\sqrt{mn}$.
Then there exists $C_{\epsilon}>0$ such that
 \begin{equation}\label{prop1est}
 \big\Vert {\LL_{\mu}(f_1,\dots,f_m)}\big\Vert_{L^{p}(\bbrn)}\le C_{\epsilon} 2^{\epsilon \mu}\Vert \Om \Vert_{L^2(\mathbb{S}^{mn-1})}\prod_{j=1}^{m}\Vert f_j\Vert_{L^{p_j}(\bbrn)}.
 \end{equation} 
\end{proposition}
The proof of the above proposition will be presented below.\\

In order to describe the induction argument, 
for $0<s<1$ and $l\in J_m$,
we define
$$\mathscr{R}^m_l(s):=\{(t_1,\dots,t_m): t_l=1 \q \text{and }~ 0\le t_j<s ~\text{ for }~ j\not= l\}.$$
and let 
$$\mathcal{C}^m(s):=\{(t_1,\dots,t_m):0<t_j<s,\q j\in J_m\}$$ 
be the open cube of side length $s$ with the lower left corner $(0,\dots,0)$.
\begin{customclaim}{$X(s)$}

 Let $1/m<p<\infty$ and $(1/p_1,\dots,1/p_m)\in\mathcal{C}^m(s) $  with $1/p_1+\dots+1/p_m=1/p.$
 Suppose that $0<\epsilon<1$ and $2^{\mu-10}>C_0\sqrt{mn}$. Then there exists $C_{\epsilon}>0$ such that
 $$\big\Vert \LL_{\mu}(f_1,\dots,f_m)   \big\Vert_{L^{p}(\bbrn)}\lesssim 2^{\epsilon \mu}\Vert \Omega\Vert_{L^{\frac{1}{1-s}}(\mathbb{S}^{mn-1})} \prod_{j=1}^{m}\Vert f_j\Vert_{L^{p_j}(\bbrn)}.$$
\end{customclaim}
\begin{customclaim}{$Y(s)$}
Let $1/m<p<1$ and $(1/p_1,\dots,1/p_m)\in \bigcup_{l=1}^{m}\mathscr{R}^m_l(s)$ with $1/p_1+\dots+1/p_m=1/p$.
Suppose that $0<\epsilon<1$ and $2^{\mu-10}>C_0\sqrt{mn}$. Then there exists $C_{\epsilon}>0$ such that
\begin{equation*}
\big\Vert \LL_{\mu}(f_1,\dots,f_m)\big\Vert_{L^{p,\infty}(\bbrn)}\le C_{\epsilon} \Vert \Omega\Vert_{L^{\frac{1}{1-s}}(\mathbb{S}^{mn-1})} 2^{\epsilon \mu}   \prod_{j=1}^{m}\Vert f_j\Vert_{L^{p_j}(\bbrn)}
\end{equation*}
\end{customclaim}
\begin{customclaim}{$Z(s)$}
Let $1/m<p<\infty$ and $(1/p_1,\dots,1/p_m)\in\bigcup_{l=1}^m\mathscr{V}_l^m(s)$  with $1/p_1+\dots+1/p_m=1/p$.
Suppose that $0<\epsilon<1$ and $2^{\mu-10}>C_0\sqrt{mn}$. Then there exists $C_{\epsilon}>0$ such that
\begin{equation*}
\big\Vert \LL_{\mu}(f_1,\dots,f_m)\big\Vert_{L^{p}(\bbrn)}\le C_{\epsilon} \Vert \Omega\Vert_{L^{\frac{1}{1-s}}(\mathbb{S}^{mn-1})} 2^{\epsilon \mu}   \prod_{j=1}^{m}\Vert f_j\Vert_{L^{p_j}(\bbrn)}
\end{equation*}
\end{customclaim}
\begin{customclaim}{$\Sigma(s)$}
Let $1/m<p<\infty$ and $(1/p_1,\dots,1/p_m)\in\mathcal{H}^m(s)$  with $1/p_1+\dots+1/p_m=1/p$.
Suppose that $0<\epsilon<1$ and $2^{\mu-10}>C_0\sqrt{mn}$. Then there exists $C_{\epsilon}>0$ such that
\begin{equation*}
\big\Vert \LL_{\mu}(f_1,\dots,f_m)\big\Vert_{L^{p}(\bbrn)}\le C_{\epsilon} \Vert \Omega\Vert_{L^{\frac{1}{1-s}}(\mathbb{S}^{mn-1})} 2^{\epsilon \mu}   \prod_{j=1}^{m}\Vert f_j\Vert_{L^{p_j}(\bbrn)}
\end{equation*}
\end{customclaim}

 \begin{figure}[h]
\begin{tikzpicture}

\path[fill=red!5] (0-\gap,0,1.95)--(1.4625-\gap,0,1.95)--(1.4625-\gap,1.4625,1.95)--(0-\gap,1.4625,1.95)--(0-\gap,0,1.95);
\path[fill=red!5] (1.4625-\gap,0,1.95)--(1.4625-\gap,0,0)--(1.4625-\gap,1.4625,0)--(1.4625-\gap,1.4625,1.95)--(1.4625-\gap,0,1.95);
\path[fill=red!5] (1.4625-\gap,1.4625,0)--(0-\gap,1.4625,0)--(0-\gap,1.4625,1.95)--(1.4625-\gap,1.4625,1.95)--(1.4625-\gap,1.4625,0);

\draw[dash pattern= { on 2pt off 1pt}](0-\gap,1.4625,0)--(1.4625-\gap,1.4625,0)--(1.4625-\gap,0,0)--(1.4625-\gap,0,1.95)--(0-\gap,0,1.95)--(0-\gap,1.4625,1.95)--(0-\gap,1.4625,0);
\draw[dash pattern= { on 2pt off 1pt}](0-\gap,1.4625,1.95)--(1.4625-\gap,1.4625,1.95)--(1.4625-\gap,0,1.95);
\draw[dash pattern= { on 2pt off 1pt}](1.4625-\gap,1.4625,1.95)--(1.4625-\gap,1.4625,0);

\draw[dotted] (1.4625-\gap,1.4625,3)--(1.4625-\gap,2.25,1.95)--(2.25-\gap,1.4625,1.95)--(1.4625-\gap,1.4625,3);
\draw[dotted] (1.4625-\gap,1.4625,3)--(1.4625-\gap,0,3)--(2.25-\gap,0,1.95)--(2.25-\gap,1.4625,1.95);
\draw[dotted] (1.4625-\gap,2.25,1.95)--(1.4625-\gap,2.25,0)--(2.25-\gap,1.4625,0)--(2.25-\gap,1.4625,1.95);
\draw[dotted](1.4625-\gap,1.4625,3)--(0-\gap,1.4625,3)--(0-\gap,2.25,1.95)--(1.4625-\gap,2.25,1.95);
\draw[dotted] (2.25-\gap,0,1.95)--(2.25-\gap,0,0)--(2.25-\gap,1.4625,0);
\draw[dotted] (0-\gap,2.25,1.95)--(0-\gap,2.25,0)--(1.4625-\gap,2.25,0);
\draw[dotted] (1.4625-\gap,0,3)--(0-\gap,0,3)--(0-\gap,1.4625,3);

\draw[dash pattern= { on 1pt off 1pt}] (0-\gap,0,0)--(1.4625-\gap,0,0);
\draw[dash pattern= { on 1pt off 1pt}] (0-\gap,0,0)--(0-\gap,1.4625,0);
\draw[dash pattern= { on 1pt off 1pt}] (0-\gap,0,0)--(0-\gap,0,1.95);
\draw [->] (0-\gap,0,1.95)--(0-\gap,0,4);
\draw [->] (0-\gap,1.4625,0)--(0-\gap,3,0);
\draw [->] (1.4625-\gap,0,0)--(3-\gap,0,0);

\node [below] at (3-\gap,0,0) {\tiny$t_1$};
\node [left] at (0-\gap,3,0) {\tiny$t_2$};
\node [below] at (0-\gap,0,4) {\tiny$t_3$};

\node [above right] at (1.4-\gap,1.4,0) {\tiny$(s,s,0)$};
\node [right] at (1.3-\gap,-0.2,1.7) {\tiny$(s,0,s)$};
\node [ left] at (0-\gap,0,1.7) {\tiny$(0,0,s)$};
\node [left] at (0-\gap,1.4,1.7) {\tiny$(0,s,s)$};
\node [right] at (1.3-\gap,1.3,1.7) {\tiny$(s,s,s)$};
\node [below right] at (1.3-\gap,0,-0.2) {\tiny$(s,0,0)$};
\node [above left] at (0-\gap,1.4,0) {\tiny$(0,s,0)$};

\node  at (0.7-\gap,0.7,1) {$\mathcal{C}^3(s)$};

\node [below] at (1.2-\gap,-1.5,1.3) {in $\mathrm{ \bf Claim}~ X(s)$};

\path[fill=red!5] (0,0,3)--(1.4625,0,3)--(1.4625,1.4625,3)--(0,1.4625,3)--(0,0,3);
\path[fill=red!5] (2.25,0,1.95)--(2.25,0,0)--(2.25,1.4625,0)--(2.25,1.4625,1.95)--(2.25,0,1.95);
\path[fill=red!5] (1.4625,2.25,0)--(0,2.25,0)--(0,2.25,1.95)--(1.4625,2.25,1.95)--(1.4625,2.25,0);

\draw[dotted] (1.4625,1.4625,3)--(1.4625,2.25,1.95)--(2.25,1.4625,1.95)--(1.4625,1.4625,3);
\draw[dotted] (1.4625,1.4625,3)--(1.4625,0,3)--(2.25,0,1.95);
\draw[dotted] (1.4625,2.25,0)--(2.25,1.4625,0)--(2.25,1.4625,1.95);
\draw[dotted] (0,1.4625,3)--(0,2.25,1.95)--(1.4625,2.25,1.95);

\draw[dash pattern= { on 2pt off 1pt}] (1.4625,0,3)--(1.4625,1.4625,3)--(0,1.4625,3);
\draw[dash pattern= { on 2pt off 1pt}]  (0,2.25,1.95)--(1.4625,2.25,1.95)--(1.4625,2.25,0);
\draw[dash pattern= { on 2pt off 1pt}] (2.25,0,1.95)--(2.25,1.4625,1.95)--(2.25,1.4625,0);

\draw[dash pattern= { on 2pt off 1pt}] (2.25,0,1.95)--(2.25,0,0)--(2.25,1.4625,0);
\draw[dash pattern= { on 2pt off 1pt}] (0,2.25,1.95)--(0,2.25,0)--(1.4625,2.25,0);
\draw[dash pattern= { on 2pt off 1pt}] (1.4625,0,3)--(0,0,3)--(0,1.4625,3);

\draw[dash pattern= { on 1pt off 1pt}] (0,0,0)--(0.3,0,0);
\draw[dash pattern= { on 1pt off 1pt}] (1.5,0,0)--(2.25,0,0);
\draw[dash pattern= { on 1pt off 1pt}] (0,0,0)--(0,0.3,0);
\draw[dash pattern= { on 1pt off 1pt}] (0,1.5,0)--(0,2.25,0);
\draw[dash pattern= { on 1pt off 1pt}] (0,0,0)--(0,0,3);
\draw [->] (0,0,3)--(0,0,4);
\draw [->] (0,2.25,0)--(0,3,0);
\draw [->] (2.25,0,0)--(3,0,0);

\draw [-] (0.3,0,0)--(1.5,0,0);
\draw [-] (0,0.3,0)--(0,1.5,0);

\node [below] at (3,0,0) {\tiny$t_1$};
\node [left] at (0,3,0) {\tiny$t_2$};
\node [below] at (0,0,4) {\tiny$t_3$};

\node  at (2.25,0.7,0.9) {$\mathscr{R}_1^3(s)$};
\node  at (0.8,2.25,0.9) {$\mathscr{R}^3_2(s)$};
\node  at (0.8,0.7,3) {$\mathscr{R}^3_3(s)$};

\node [left] at (0,2.25,-0.2) {\tiny$(0,1,0)$};
\node [right] at (1.38,2.25,-0.2) {\tiny$(s,1,0)$};
\node [left] at (0,2.25,1.8) {\tiny$(0,1,s)$};
\node [right] at (1.45,2.25,1.9) {\tiny$(s,1,s)$};

\node [right] at (2.1,1.4,-0.2) {\tiny$(1,s,0)$};
\node [right] at (2.1,0.1,-0.2) {\tiny$(1,0,0)$};
\node [right] at (2.25,0,1.95) {\tiny$(1,0,s)$};

\node [right] at (1.1,-0.2,3) {\tiny$(s,0,1)$};
\node [left] at (0,0,2.8) {\tiny$(0,0,1)$};
\node [left] at (0,1.4625,2.8) {\tiny$(0,s,1)$};

\node [below] at (1.2,-1.5,1.3) {in $\mathrm{ \bf Claim}~ Y(s)$};


\path[fill=red!5] (0-\gap,0-\gaptt,1.95)--(0-\gap,2.25-\gaptt,1.95)--(0-\gap,2.25-\gaptt,0)--(1.4625-\gap,2.25-\gaptt,0)--(1.4625-\gap,0-\gaptt,-0.2)--(1.4625-\gap,0-\gaptt,1.95)--(0-\gap,0-\gaptt,1.95);
\path[fill=red!5] (0-\gap,1.4-\gaptt,-0.2)--(2.2-\gap,1.4-\gaptt,-0.2)--(2.2-\gap,0-\gaptt,-0.2)--(2.25-\gap,0-\gaptt,1.95)--(0-\gap,0-\gaptt,1.95);
\path[fill=red!5] (0-\gap,1.4625-\gaptt,-0.2)--(0-\gap,1.4625-\gaptt,3)--(0-\gap,0-\gaptt,3)--(1.4625-\gap,0-\gaptt,3)--(1.4625-\gap,0-\gaptt,-0.2);

\draw[dotted](1.4625-\gap,1.4625-\gaptt,3)--(1.4625-\gap,2.25-\gaptt,1.95)--(2.25-\gap,1.4625-\gaptt,1.95)--(1.4625-\gap,1.4625-\gaptt,3);
\draw[dotted] (1.4625-\gap,1.4625-\gaptt,3)--(1.4625-\gap,0-\gaptt,3)--(2.25-\gap,0-\gaptt,1.95);
\draw[dotted] (1.4625-\gap,2.25-\gaptt,0)--(2.25-\gap,1.4625-\gaptt,0);
\draw[dotted] (1.4625-\gap,1.4625-\gaptt,3)--(0-\gap,1.4625-\gaptt,3)--(0-\gap,2.25-\gaptt,1.95);
\draw[dash pattern= { on 2pt off 1pt}] (2.25-\gap,0-\gaptt,1.95)--(2.25-\gap,0-\gaptt,0)--(2.25-\gap,1.4625-\gaptt,0)--(2.25-\gap,1.4625-\gaptt,1.95)--(2.25-\gap,0-\gaptt,1.95);
\draw[dash pattern= { on 2pt off 1pt}] (0-\gap,2.25-\gaptt,1.95)--(0-\gap,2.25-\gaptt,0)--(1.4625-\gap,2.25-\gaptt,0)--(1.4625-\gap,2.25-\gaptt,1.95)--(0-\gap,2.25-\gaptt,1.95);
\draw[dash pattern= { on 2pt off 1pt}] (1.4625-\gap,0-\gaptt,3)--(0-\gap,0-\gaptt,3)--(0-\gap,1.4625-\gaptt,3)--(1.4625-\gap,1.4625-\gaptt,3)--(1.4625-\gap,0-\gaptt,3);

\draw[dash pattern= { on 1pt off 1pt}] (0-\gap,0-\gaptt,0)--(2.25-\gap,0-\gaptt,0);
\draw[dash pattern= { on 1pt off 1pt}] (0-\gap,0-\gaptt,0)--(0-\gap,2.25-\gaptt,0);
\draw[dash pattern= { on 1pt off 1pt}] (0-\gap,0-\gaptt,0)--(0-\gap,0-\gaptt,3);
\draw [->] (0-\gap,0-\gaptt,3)--(0-\gap,0-\gaptt,4);
\draw [->] (0-\gap,2.25-\gaptt,0)--(0-\gap,3-\gaptt,0);
\draw [->] (2.25-\gap,0-\gaptt,0)--(3-\gap,0-\gaptt,0);

\node [below] at (3-\gap,0-\gaptt,0) {\tiny$t_1$};
\node [left] at (0-\gap,3-\gaptt,0) {\tiny$t_2$};
\node [below] at (0-\gap,0-\gaptt,4) {\tiny$t_3$};

\node [left] at (0-\gap,2.25-\gaptt,-0.2) {\tiny$(0,1,0)$};
\node [right] at (1.38-\gap,2.25-\gaptt,-0.2) {\tiny$(s,1,0)$};
\node [left] at (0-\gap,2.25-\gaptt,1.8) {\tiny$(0,1,s)$};

\draw[dash pattern= { on 2pt off 1pt}] (0-\gap,1.4-\gaptt,-0.2)--(2.15-\gap,1.4-\gaptt,-0.2);
\draw[dash pattern= { on 2pt off 1pt}] (0-\gap,1.4625-\gaptt,1.95)--(2.25-\gap,1.4625-\gaptt,1.95);
\draw[dash pattern= { on 2pt off 1pt}] (0-\gap,0-\gaptt,1.95)--(2.25-\gap,0-\gaptt,1.95);

\draw[dash pattern= { on 2pt off 1pt}] (0-\gap,1.4625-\gaptt,0)--(0-\gap,1.4625-\gaptt,3);
\draw[dash pattern= { on 2pt off 1pt}] (1.4625-\gap,1.4625-\gaptt,0)--(1.4625-\gap,1.4625-\gaptt,3);
\draw[dash pattern= { on 2pt off 1pt}] (1.4625-\gap,0-\gaptt,0)--(1.4625-\gap,0-\gaptt,3);

\draw[dash pattern= { on 2pt off 1pt}] (1.38-\gap,0-\gaptt,-0.2)--(1.38-\gap,2.2-\gaptt,-0.2);
\draw[dash pattern= { on 2pt off 1pt}] (1.4625-\gap,0-\gaptt,1.95)--(1.4625-\gap,2.25-\gaptt,1.95);
\draw[dash pattern= { on 2pt off 1pt}] (0-\gap,0-\gaptt,1.95)--(0-\gap,2.25-\gaptt,1.95);

\node [right] at (2.1-\gap,1.4-\gaptt,-0.2) {\tiny$(1,s,0)$};
\node [right] at (2.1-\gap,0.1-\gaptt,-0.2) {\tiny$(1,0,0)$};
\node [right] at (2.25-\gap,0-\gaptt,1.95) {\tiny$(1,0,s)$};

\node [right] at (1.1-\gap,-0.2-\gaptt,3) {\tiny$(s,0,1)$};
\node [left] at (0-\gap,0-\gaptt,2.8) {\tiny$(0,0,1)$};
\node [left] at (0-\gap,1.4625-\gaptt,2.8) {\tiny$(0,s,1)$};

\node  at (2.1-\gap,0.7-\gaptt,0.9) {$\mathscr{V}_1^3(s)$};
\node  at (0.8-\gap,1.9-\gaptt,0.9) {$\mathscr{V}^3_2(s)$};
\node  at (0.8-\gap,0.7-\gaptt,2.8) {$\mathscr{V}^3_3(s)$};

\node [below] at (1.2-\gap,-1.5-\gaptt,1.3) {in $\mathrm{ \bf Claim}~ Z(s)$};


\path[fill=red!5] (0,0-\gaptt,3)--(1.4625,0-\gaptt,3)--(2.25,0-\gaptt,1.95)--(2.25,0-\gaptt,0)--(2.25,1.4625-\gaptt,0)--(1.4625,2.25-\gaptt,0)--(0,2.25-\gaptt,0)--(0,2.25-\gaptt,1.95)--(0,1.4625-\gaptt,3)--(0,0-\gaptt,3);

\draw[dash pattern= { on 2pt off 1pt}](1.4625,1.4625-\gaptt,3)--(1.4625,2.25-\gaptt,1.95)--(2.25,1.4625-\gaptt,1.95)--(1.4625,1.4625-\gaptt,3);
\draw[dash pattern= { on 2pt off 1pt}] (1.4625,1.4625-\gaptt,3)--(1.4625,0-\gaptt,3)--(2.25,0-\gaptt,1.95)--(2.25,1.4625-\gaptt,1.95);
\draw[dash pattern= { on 2pt off 1pt}] (1.4625,2.25-\gaptt,1.95)--(1.4625,2.25-\gaptt,0)--(2.25,1.4625-\gaptt,0)--(2.25,1.4625-\gaptt,1.95);
\draw[dash pattern= { on 2pt off 1pt}] (1.4625,1.4625-\gaptt,3)--(0,1.4625-\gaptt,3)--(0,2.25-\gaptt,1.95)--(1.4625,2.25-\gaptt,1.95);
\draw[dash pattern= { on 2pt off 1pt}] (2.25,0-\gaptt,1.95)--(2.25,0-\gaptt,0)--(2.25,1.4625-\gaptt,0);
\draw[dash pattern= { on 2pt off 1pt}] (0,2.25-\gaptt,1.95)--(0,2.25-\gaptt,0)--(1.4625,2.25-\gaptt,0);
\draw[dash pattern= { on 2pt off 1pt}] (1.4625,0-\gaptt,3)--(0,0-\gaptt,3)--(0,1.4625-\gaptt,3);

\draw[dash pattern= { on 1pt off 1pt}] (0,0-\gaptt,0)--(2.25,0-\gaptt,0);
\draw[dash pattern= { on 1pt off 1pt}] (0,0-\gaptt,0)--(0,2.25-\gaptt,0);
\draw[dash pattern= { on 1pt off 1pt}] (0,0-\gaptt,0)--(0,0-\gaptt,3);
\draw [->] (0,0-\gaptt,3)--(0,0-\gaptt,4);
\draw [->] (0,2.25-\gaptt,0)--(0,3-\gaptt,0);
\draw [->] (2.25,0-\gaptt,0)--(3,0-\gaptt,0);

\node [below] at (3,0-\gaptt,0) {\tiny$t_1$};
\node [left] at (0,3-\gaptt,0) {\tiny$t_2$};
\node [below] at (0,0-\gaptt,4) {\tiny$t_3$};

\node [left] at (0,2.25-\gaptt,-0.2) {\tiny$(0,1,0)$};
\node [right] at (1.38,2.25-\gaptt,-0.2) {\tiny$(s,1,0)$};
\node [left] at (0,2.25-\gaptt,1.8) {\tiny$(0,1,s)$};
\node [right] at (1.45,2.25-\gaptt,1.9) {\tiny$(s,1,s)$};

\node [right] at (2.1,1.4-\gaptt,-0.2) {\tiny$(1,s,0)$};
\node [right] at (2.1,0.1-\gaptt,-0.2) {\tiny$(1,0,0)$};
\node [right] at (2.25,0-\gaptt,1.95) {\tiny$(1,0,s)$};

\node [right] at (1.1,-0.2-\gaptt,3) {\tiny$(s,0,1)$};
\node [left] at (0,0-\gaptt,2.8) {\tiny$(0,0,1)$};
\node [left] at (0,1.4625-\gaptt,2.8) {\tiny$(0,s,1)$};

\node  at (1.2,1.15-\gaptt,1.5) {$\HH^3(s)$};

\node [below] at (1.2,-1.5-\gaptt,1.3) {in $\mathrm{ \bf Claim}~ \Sigma(s)$};

\end{tikzpicture}
\caption{The trilinear case $m=3$ : the range of $(\frac{1}{p_1},\frac{1}{p_2},\frac{1}{p_3})$}\label{fig1}
\end{figure}
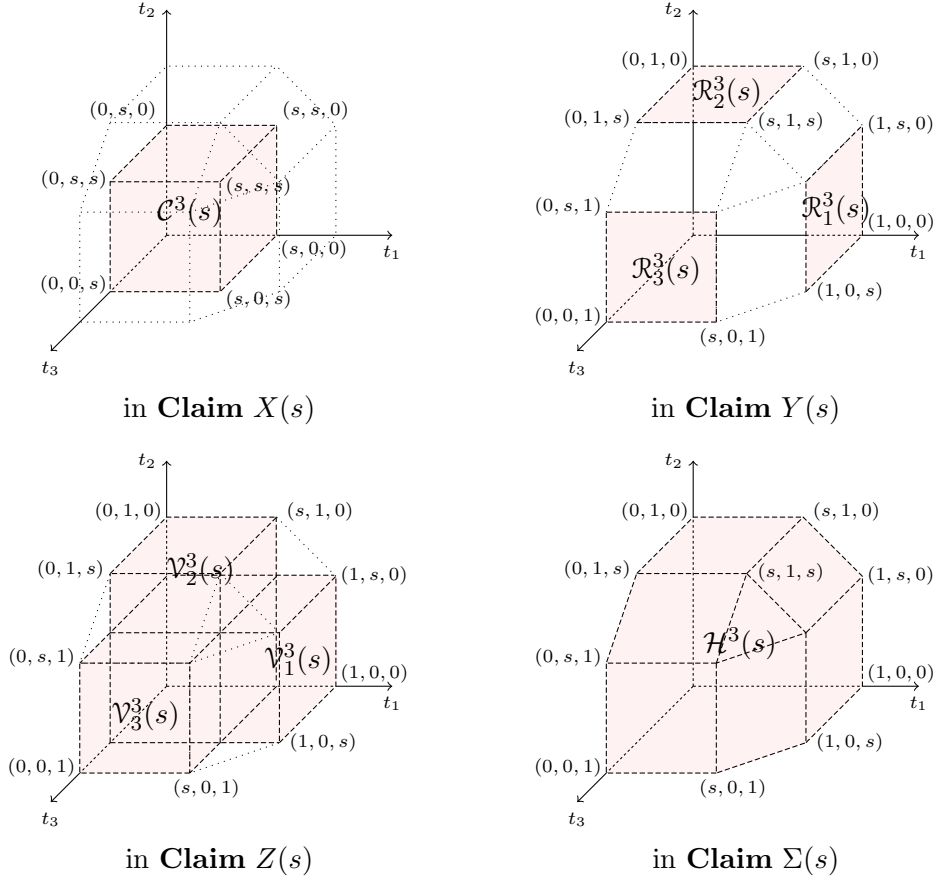
Then the following proposition will play an essential role in the induction steps.
\begin{proposition}\label{inductionpropo}
Let $0< s<1$.
Then
$$ \mathrm{ \bf Claim}~ X(s) \Rightarrow  \mathrm{ \bf Claims}~X(s)\text{ and } Y(s) \Rightarrow \mathrm{ \bf Claim}~ Z(s) \Rightarrow \mathrm{ \bf Claim}~ \Sigma(s).$$
\end{proposition}
The proof of Proposition \ref{inductionpropo} will be given below.\\

We now complete the proof of Proposition \ref{mainproposition}, using Propositions \ref{lppropo} and \ref{inductionpropo}.

\begin{proof}[Proof of Proposition~\ref{mainproposition}]
For $\nu\in \bbn_0$, let
$$a_{\nu}:=1-\frac{1}{2}\Big(1-\frac{1}{m} \Big)^{\nu},$$
for which $(a_{\nu+1},\dots,a_{\nu+1})\in \bbr^m$ is the center of the $(m-1)$ simplex with $m$ vertices 
$(1,a_{\nu},a_{\nu},\dots,a_{\nu})$, $(a_{\nu},1,a_{\nu},\dots, a_{\nu})$, $\dots$, $(a_{\nu},\dots, a_{\nu},1,a_{\nu})$, and $(a_{\nu},\dots, a_{\nu},a_{\nu},1)$.
 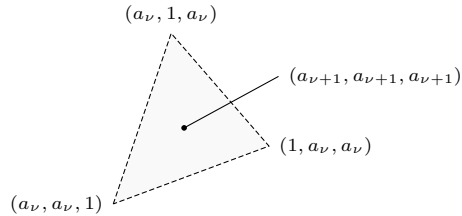
\begin{figure}[h]
\begin{tikzpicture}

\path[fill=gray!5] (1.7,1.5,4)--(1.7,3,2)--(3,1.5,2)--(1.7,1.5,4);

\draw[dash pattern= { on 2pt off 1pt}] (1.7,1.5,4)--(1.7,3,2)--(3,1.5,2)--(1.7,1.5,4);

\filldraw[fill=black] (2.13,2,2.67)  circle[radius=0.3mm];
\draw [-] (2.13,2,2.67)--(3,2.3,1.67);
\node [right] at (3,2.3,1.67) {\tiny$(a_{\nu+1},a_{\nu+1},a_{\nu+1})$};

\node [right] at (3,1.5,2) {\tiny$(1,a_{\nu},a_{\nu})$};
\node [above] at (1.7,3,2) {\tiny$(a_{\nu},1,a_{\nu})$};
\node [left] at (1.7,1.5,4) {\tiny$(a_{\nu},a_{\nu},1)$};

\end{tikzpicture}
\caption{$(a_{\nu+1},a_{\nu+1},a_{\nu+1})$ when $m=3$}\label{fig0}
\end{figure}
We notice that $a_0=1/2$, $a_{\nu}\nearrow 1$, and
$a_{\nu+1}=\frac{a_{\nu}(m-1)+1}{m}.$
Moreover, by definition, we have
$$\mathcal{C}^m(a_{\nu+1})\subset \HH^m(a_{\nu}) \q \text{ for all }~\nu\in \bbn_0,$$
which implies
\begin{equation}\label{zanutoxanu1}
 \mathrm{ \bf Claim}~ \Sigma(a_{\nu})\Rightarrow \mathrm{ \bf Claim}~ X(a_{\nu+1})  \q \text{ for all }~\nu\in \bbn_0
 \end{equation}
 as $L^{\frac{1}{1-a_{\nu+1}}}(\mathbb{S}^{mn-1})\hookrightarrow L^{\frac{1}{1-a_{\nu}}}(\mathbb{S}^{mn-1})$.
  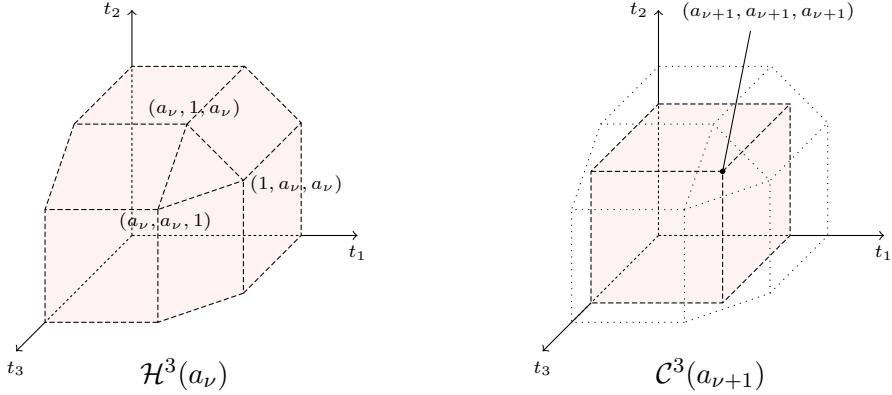
\begin{figure}[h]
\begin{tikzpicture}

\path[fill=red!5] (0+\gapt,0,3)--(1.5+\gapt,0,3)--(2.25+\gapt,0,1.995)--(2.25+\gapt,0,0)--(2.25+\gapt,1.5,0)--(1.5+\gapt,2.25,0)--(0+\gapt,2.25,0)--(0+\gapt,2.25,1.995)--(0+\gapt,1.5,3)--(0+\gapt,0,3);

\draw[dash pattern= { on 2pt off 1pt}](1.5+\gapt,1.5,3)--(1.5+\gapt,2.25,1.995)--(2.25+\gapt,1.5,1.995)--(1.5+\gapt,1.5,3);
\draw[dash pattern= { on 2pt off 1pt}] (1.5+\gapt,1.5,3)--(1.5+\gapt,0,3)--(2.25+\gapt,0,1.995)--(2.25+\gapt,1.5,1.995);
\draw[dash pattern= { on 2pt off 1pt}] (1.5+\gapt,2.25,1.995)--(1.5+\gapt,2.25,0)--(2.25+\gapt,1.5,0)--(2.25+\gapt,1.5,1.995);
\draw[dash pattern= { on 2pt off 1pt}] (1.5+\gapt,1.5,3)--(0+\gapt,1.5,3)--(0+\gapt,2.25,1.995)--(1.5+\gapt,2.25,1.995);
\draw[dash pattern= { on 2pt off 1pt}] (2.25+\gapt,0,1.995)--(2.25+\gapt,0,0)--(2.25+\gapt,1.5,0);
\draw[dash pattern= { on 2pt off 1pt}] (0+\gapt,2.25,1.995)--(0+\gapt,2.25,0)--(1.5+\gapt,2.25,0);
\draw[dash pattern= { on 2pt off 1pt}] (1.5+\gapt,0,3)--(0+\gapt,0,3)--(0+\gapt,1.5,3);

\node [right] at (2.2+\gapt,1.45,1.995) {\tiny$(1,a_{\nu},a_{\nu})$};
\node [above] at (1.6+\gapt,2.2,2) {\tiny$(a_{\nu},1,a_{\nu})$};
\node [below] at (1.6+\gapt,1.6,3) {\tiny$(a_{\nu},a_{\nu},1)$};

\draw[dash pattern= { on 1pt off 1pt}] (0+\gapt,0,0)--(2.25+\gapt,0,0);
\draw[dash pattern= { on 1pt off 1pt}] (0+\gapt,0,0)--(0+\gapt,2.25,0);
\draw[dash pattern= { on 1pt off 1pt}] (0+\gapt,0,0)--(0+\gapt,0,3);
\draw [->] (0+\gapt,0,3)--(0+\gapt,0,4);
\draw [->] (0+\gapt,2.25,0)--(0+\gapt,3,0);
\draw [->] (2.25+\gapt,0,0)--(3+\gapt,0,0);

\node [below] at (3+\gapt,0,0) {\tiny$t_1$};
\node [left] at (0+\gapt,3,0) {\tiny$t_2$};
\node [below] at (0+\gapt,0,4) {\tiny$t_3$};

\node [below] at (1.2+\gapt,-1,1.3) {$\HH^3(a_{\nu})$};


\path[fill=red!5] (0+\gapt+\gapt,0,2.33)--(1.75+\gapt+\gapt,0,2.33)--(1.75+\gapt+\gapt,1.75,2.33)--(0+\gapt+\gapt,1.75,2.33)--(0+\gapt+\gapt,0,2.33);
\path[fill=red!5] (1.75+\gapt+\gapt,0,2.33)--(1.75+\gapt+\gapt,0,0)--(1.75+\gapt+\gapt,1.75,0)--(1.75+\gapt+\gapt,1.75,2.33)--(1.75+\gapt+\gapt,0,2.33);
\path[fill=red!5] (1.75+\gapt+\gapt,1.75,0)--(0+\gapt+\gapt,1.75,0)--(0+\gapt+\gapt,1.75,2.33)--(1.75+\gapt+\gapt,1.75,2.33)--(1.75+\gapt+\gapt,1.75,0);

\draw[dash pattern= { on 2pt off 1pt}] (0+\gapt+\gapt,1.75,0)--(1.75+\gapt+\gapt,1.75,0)--(1.75+\gapt+\gapt,0,0)--(1.75+\gapt+\gapt,0,2.33)--(0+\gapt+\gapt,0,2.33)--(0+\gapt+\gapt,1.75,2.33)--(0+\gapt+\gapt,1.75,0);
\draw[dash pattern= { on 2pt off 1pt}] (0+\gapt+\gapt,1.75,2.33)--(1.75+\gapt+\gapt,1.75,2.33)--(1.75+\gapt+\gapt,0,2.33);
\draw[dash pattern= { on 2pt off 1pt}] (1.75+\gapt+\gapt,1.75,2.33)--(1.75+\gapt+\gapt,1.75,0);

\draw[dotted] (1.5+\gapt+\gapt,1.5,3)--(1.5+\gapt+\gapt,2.25,1.995)--(2.25+\gapt+\gapt,1.5,1.995)--(1.5+\gapt+\gapt,1.5,3);
\draw[dotted] (1.5+\gapt+\gapt,1.5,3)--(1.5+\gapt+\gapt,0,3)--(2.25+\gapt+\gapt,0,1.995)--(2.25+\gapt+\gapt,1.5,1.995);
\draw[dotted] (1.5+\gapt+\gapt,2.25,1.995)--(1.5+\gapt+\gapt,2.25,0)--(2.25+\gapt+\gapt,1.5,0)--(2.25+\gapt+\gapt,1.5,1.995);
\draw[dotted](1.5+\gapt+\gapt,1.5,3)--(0+\gapt+\gapt,1.5,3)--(0+\gapt+\gapt,2.25,1.995)--(1.5+\gapt+\gapt,2.25,1.995);
\draw[dotted] (2.25+\gapt+\gapt,0,1.995)--(2.25+\gapt+\gapt,0,0)--(2.25+\gapt+\gapt,1.5,0);
\draw[dotted] (0+\gapt+\gapt,2.25,1.995)--(0+\gapt+\gapt,2.25,0)--(1.5+\gapt+\gapt,2.25,0);
\draw[dotted] (1.5+\gapt+\gapt,0,3)--(0+\gapt+\gapt,0,3)--(0+\gapt+\gapt,1.5,3);

\draw[dash pattern= { on 1pt off 1pt}] (0+\gapt+\gapt,0,0)--(2.25+\gapt+\gapt,0,0);
\draw[dash pattern= { on 1pt off 1pt}] (0+\gapt+\gapt,0,0)--(0+\gapt+\gapt,2.25,0);
\draw[dash pattern= { on 1pt off 1pt}] (0+\gapt+\gapt,0,0)--(0+\gapt+\gapt,0,3);
\draw [->] (0+\gapt+\gapt,0,2.33)--(0+\gapt+\gapt,0,4);
\draw [->] (0+\gapt+\gapt,1.75,0)--(0+\gapt+\gapt,3,0);
\draw [->] (1.75+\gapt+\gapt,0,0)--(3+\gapt+\gapt,0,0);

\node [below] at (3+\gapt+\gapt,0,0) {\tiny$t_1$};
\node [left] at (0+\gapt+\gapt,3,0) {\tiny$t_2$};
\node [below] at (0+\gapt+\gapt,0,4) {\tiny$t_3$};

\filldraw[fill=black] (1.75+\gapt+\gapt,1.75,2.33)  circle[radius=0.3mm];
\draw [-] (1.75+\gapt+\gapt,1.75,2.33)--(1.05+\gapt+\gapt,2.55,-0.45);
\node [above right] at (0+\gapt+\gapt,2.55,-0.45) {\tiny$(a_{\nu+1},a_{\nu+1},a_{\nu+1})$};

\node [below] at (1.2+\gapt+\gapt,-1,1.3) {$\mathcal{C}^3(a_{\nu+1})$};

\end{tikzpicture}
\caption{The trilinear case $m=3$ : $\HH^3(a_{\nu})$ and $\mathcal{C}^3(a_{\nu+1})$}\label{fig5}
\end{figure}
Now Proposition \ref{lppropo} implies that $\mathrm{\bf Claim} ~X(a_0)$ holds, and consequently, $\mathrm{\bf Claim} ~\Sigma(a_{\nu})$ should be also true for all $\nu\in\bbn_0$ with the aid of Proposition \ref{inductionpropo} and (\ref{zanutoxanu1}).

When $s=1/2~(=a_0)$, the asserted estimate (\ref{mainpropoest}) is exactly $\mathrm{ \bf Claim}~ \Sigma(a_{0})$.
If $a_{\nu}<s\le a_{\nu+1}$ for some $\nu\in\bbn_0$, then $\mathcal{C}^m(s)\subset \mathcal{H}^m(a_{\nu})$. This yields that
$ \mathrm{ \bf Claim}~ X(s)$ holds since $L^{\frac{1}{1-s}}(\mathbb{S}^{mn-1})\hookrightarrow L^{\frac{1}{1-a_{\nu}}}(\mathbb{S}^{mn-1})$,
and accordingly, Proposition \ref{inductionpropo} shows that $ \mathrm{ \bf Claim}~ \Sigma (s)$ works. This finishes the proof of Proposition \ref{mainproposition}.\\

\end{proof}

Now let us prove Propositions \ref{lppropo} and \ref{inductionpropo}.
\begin{proof}[Proof of Proposition \ref{lppropo}]
 Observing that $\wh{K_{\mu}^0}\in L^{2}((\bbrn)^m)$, we apply the wavelet decomposition (\ref{daubechewavelet})  to write 
\begin{equation}\label{kmuexpansion}
\wh{K_{\mu}^0}(\xxxi)=\sum_{\la\in\bbn_0}\sum_{\GGG\in\II^{\la}}\sum_{\kkk\in (\bbzn)^m}b_{\GGG,\kkk}^{\la,\mu}\Psi_{G_1,k_1}^{\la}(\xi_1)\cdots \Psi_{G_m,k_m}^{\la}(\xi_m)
\end{equation}
where 
\begin{equation*}
b_{\GGG,\kkk}^{\la,\mu}:=\int_{(\bbrn)^m}{\wh{K_{\mu}^0}(\xxxi)\Psi_{\GGG,\kkk}^{\la}(\xxxi)}d\xxxi.
\end{equation*}
It is known in \cite[Lemma 7]{Gr_He_Ho2018} that for any $0<\delta<1/2$,
\begin{equation}\label{maininftyest}
\big\Vert \{b_{\GGG,\kkk}^{\la,\mu}\}_{\kkk\in (\bbzn)^m}\big\Vert_{\ell^{\infty}}\lesssim      2^{-\delta {\mu}}2^{-\la (L+1+mn)} \Vert \Om\Vert_{L^2(\mathbb{S}^{mn-1})}
\end{equation} where $L$ is the number of vanishing moments of $\Psi_{\GGG}$; this number $L$ can be chosen sufficiently large.
Moreover, it follows from the inequality (\ref{lqestimate}) and Plancherel's identity that
\begin{align}\label{mainlqest}
\big\Vert \{b_{\GGG,\kkk}^{\la,\mu}\}_{\kkk\in (\bbzn)^m}\big\Vert_{\ell^{2}}&\lesssim \big\Vert \wh{K_{\mu}^0} \big\Vert_{L^{2}((\bbrn)^m)}\lesssim \Vert \Omega \Vert_{L^2(\mathbb{S}^{mn-1})}.
\end{align} 

Using \eqref{kernelcharacter} and \eqref{kmuexpansion}, we can write
\begin{align}\label{lmusums}
\LL_{\mu}\big(f_1,\dots,f_m\big)(x)&=\sum_{\ga\in\bbz}\int_{(\bbrn)^m} 2^{\ga mn}K_{\mu}^0(2^{\ga}\yyy)\prod_{j=1}^{m}f_j(x-y_j)\, d\yyy \nonumber\\
&=\sum_{\ga\in\bbz}\int_{(\bbrn)^m}\wh{K_{\mu}^0}(\xxxi/2^{\ga})e^{2\pi i\langle x,\xi_1+\dots+\xi_m\rangle}\prod_{j=1}^{m}\wh{f_j}(\xi_j) \;d\xxxi \nonumber\\
&= \sum_{\la\in\bbn_0}\sum_{\GGG\in\II^{\la}} \sum_{\ga\in\bbz}\sum_{\kkk\in (\bbzn)^m} b_{\GGG,\kkk}^{\la,\mu}\prod_{j=1}^{m} L_{G_j,k_j}^{\la,\ga}f_j (x)
\end{align}
where $L_{G,k}^{\la,\ga}$ is defined in \eqref{lgklg}.

When $2^{{\mu}-10}>C_0\sqrt{mn}$, we may replace $\sum_{\kkk\in (\bbzn)^m}$ in \eqref{lmusums} by $\sum_{2^{\la+\mu-2}\le |\kkk|\le 2^{\la+\mu+2}}$
, due to the compact supports of $\wh{K_{\mu}^0}$ and $\Psi_{\GGG,\kkk}^{\la}$. 
In addition, by symmetry, it suffices to focus only on the case $|k_1|\ge\cdots\ge |k_m|$.
Therefore the estimate (\ref{prop1est}) can be reduced to the inequality
\begin{equation}\label{2mmainest}
\bigg\Vert  \sum_{\la\in\bbn_0}\sum_{\GGG\in\II^{\la}} \sum_{\ga\in\bbz}\sum_{\kkk\in\UU^{\la+\mu}} b_{\GGG,\kkk}^{\la,\mu}\prod_{j=1}^{m} L_{G_j,k_j}^{\la,\ga}f_j      \bigg\Vert_{L^{p}(\bbrn)}\lesssim_{\epsilon} 2^{\epsilon \mu}\Vert \Omega\Vert_{L^2(\mathbb{S}^{mn-1})}\prod_{j=1}^{m}\Vert f_j\Vert_{L^{p_j}(\bbrn)}
\end{equation} 
where
\begin{equation*}
\UU^{\la+{\mu}}:=\big\{ \kkk\in (\bbzn)^m: 2^{\la+{\mu}-2}\le |\kkk|\le 2^{\la+{\mu}+2},~  |k_1|\ge\cdots \ge |k_m| \big\}.
\end{equation*}

\hfill

We split $\UU^{\la+{\mu}}$ into the following $m$ disjoint subsets:
\begin{equation*}
\UU^{\la+{\mu}}_1:=\{\kkk\in \UU^{\la+{\mu}}:|k_1|\ge 2C_0\sqrt{n}>|k_2|\ge\cdots\ge |k_m|      \}
\end{equation*}
\begin{equation*}
\UU^{\la+{\mu}}_2:=\{\kkk\in \UU^{\la+{\mu}}:|k_1|\ge |k_2|\ge 2C_0\sqrt{n}>|k_3|\ge\cdots\ge |k_m|      \}
\end{equation*}
$$\vdots$$
\begin{equation*}
\UU^{\la+{\mu}}_m:=\{\kkk\in \UU^{\la+{\mu}}:|k_1|\ge \cdots\ge |k_m|\ge 2C_0\sqrt{n}      \}.
\end{equation*}
Then the left-hand side of (\ref{2mmainest}) is estimated by
\begin{equation*}
\bigg( \sum_{l=1}^{m}\sum_{\la\in\bbn_0}\sum_{\GGG\in\II^{\la}} \Big\Vert  \sum_{\ga\in\bbz} \TT_{\GGG,l}^{\la,\ga,\mu}(f_1,\dots,f_m)    \Big\Vert_{L^{p}(\bbrn)}^{\min{\{1,p\}}}\bigg)^{1/\min{\{1,p\}}}
\end{equation*} 
where the operator $\TT_{\GGG,l}^{\la,\ga,\mu}$ is defined by
\begin{equation*}
\TT_{\GGG,l}^{\la,\ga,\mu}\big(f_1,\dots,f_m\big):=\sum_{\kkk\in\UU_l^{\la+{\mu}}}b_{\GGG,\kkk}^{\la,\mu} \prod_{j=1}^{m}L_{G_j,k_j}^{\la,\ga}f_j.
\end{equation*}
We claim that for each $l\in J_m$, there exists $M_0>0$, depending on $p_1,\dots,p_m$, such that
\begin{equation}\label{2mmaingoal}
 \Big\Vert  \sum_{\ga\in\bbz} \TT_{\GGG,l}^{\la,\ga,\mu}(f_1,\dots,f_m)      \Big\Vert_{L^{p}(\bbrn)} \lesssim_{\epsilon} 2^{\epsilon \mu}2^{-\la M_0}\Vert \Omega\Vert_{L^2(\mathbb{S}^{mn-1})}\prod_{j=1}^{m}\Vert f_j\Vert_{L^{p_j}(\bbrn)},
\end{equation} 
 which clearly concludes (\ref{2mmainest}). Therefore it remains to prove \eqref{2mmaingoal}.

\hfill

The proof of (\ref{2mmaingoal}) for the case $l=1$ relies on the fact that
 if $\widehat{g_{\ga}}$ is supported in the set $\{\xi\in \rn : C^{-1} 2^{\ga+\mu}\leq |\xi|\leq C2^{\ga+\mu}\}$ for some $C>1$ and $\mu\in \bbz$, then
\begin{equation}\label{marshallest1}
\bigg\Vert \Big\{ \La_j \Big(\sum_{\ga\in\bbz}{g_{\ga}}\!\Big)\!\Big\}_{\! j\in\mathbb{Z}}\bigg\Vert_{L^p(\ell^q)}\lesssim_{C} \big\Vert \big\{ g_j\big\}_{j\in\mathbb{Z}}\big\Vert_{L^p(\ell^q)} \q \text{uniformly in }~\mu
\end{equation}  for $0<p<\infty$ and $0<q\le \infty$. The proof of (\ref{marshallest1}) is elementary and standard, so it is omitted here. See \cite[(3.9)]{Gr_Park2021} and  \cite[Theorem 3.6]{Ya1986}  for a related argument.
Note that if $\kkk\in\UU_{1}^{\la+\mu}$ and $2^{\mu-10}\ge C_0\sqrt{mn}$, then
\begin{equation*}
2^{\la+{\mu}-3}\le 2^{\la+{\mu}-2}-2C_0\sqrt{mn}\le |\kkk|- (|k_2|^2+\dots+|k_m|^2)^{1/2}\le |k_1|\le 2^{\la+{\mu}+2},
\end{equation*}
and this implies that 
\begin{equation*}
\supp\big(\Psi_{G_1,k_1}^{\la}(\cdot/2^{\ga}) \big)\subset \{\xi\in \bbrn: 2^{\ga+{\mu}-4}\le |\xi|\le 2^{\ga+{\mu}+3}\}.
\end{equation*}
Moreover,   since $|k_j|\le 2C_0\sqrt{n}$  for $2\le j\le m$ and $2^{{\mu}-10}>C_0\sqrt{mn}$,
\begin{equation*}
\supp\big(\Psi_{G_j,k_j}^{\la}(\cdot/2^{\ga})\big)\subset \{\xi\in \bbrn: |\xi|\le m^{-1/2}2^{\ga+{\mu}-8}\}.
\end{equation*} 
Therefore, the Fourier transform of $\TT_{\GGG,1}^{\la,\ga,\mu}\big(f_1,\dots,f_m\big)$ for $2^{\mu-10}\ge C_0\sqrt{mn}$  is supported in the set $\{\xi\in \bbrn: 2^{\ga+{\mu}-5}\le |\xi|\le 2^{\ga+{\mu}+4} \}$.
Now, using  the Littlewood-Paley theory for Hardy spaces, we have
\begin{equation*}
\Big\Vert  \sum_{\ga\in\bbz} \TT_{\GGG,1}^{\la,\ga,\mu}\big(f_1,\dots,f_m\big) \Big\Vert_{L^{p}(\bbrn)}\sim \bigg\Vert  \Big\{    \La_j \Big( \sum_{\ga\in\bbz} \TT_{\GGG,1}^{\la,\ga,\mu}\big(f_1,\dots,f_m\big) \Big)\Big\}_{j\in\bbz} \Big\Vert_{L^{p}(\ell^2)}
\end{equation*}
and then (\ref{marshallest1}) yields that the above $L^{p}(\ell^{2})$-norm is dominated by a constant multiple of 
\begin{equation}\label{LHSterm}
\bigg\Vert \Big( \sum_{\ga\in\bbz}{\big| \TT_{\GGG,1}^{\la,\ga,\mu}\big(f_1,\dots,f_m\big)  \big|^2}\Big)^{1/2}\bigg\Vert_{L^{p}(\bbrn)}.
\end{equation}
Using \eqref{disjointdecom1} and \eqref{lmaximalbound}, we see that
\begin{align*}
&\big| \TT_{\GGG,1}^{\la,\ga,\mu}\big(f_1,\dots,f_m\big)(x)  \big|\\
&\le \sum_{\kkk^{*1}\in\PP_{*1}\UU_1^{\la+\mu}} \prod_{j=2}^{m}\big|L_{G_{j},k_j}^{\la,\ga}f_j(x) \big| \Big| \sum_{k_1\in Col_{\kkk^{*1}}^{\UU_1^{\la+\mu}}}b_{\GGG,\kkk}^{\la,\mu} L_{G_1,k_1}^{\la,\ga}f_1(x)\Big|\\
&\lesssim 2^{\la n(m-1)/2}\prod_{j=2}^{m}\mathcal{M}f_j(x) \sum_{\kkk^{*1}\in\PP_{*1}\UU_1^{\la+\mu}} \Big| \sum_{k_1\in Col_{\kkk^{*1}}^{\UU_1^{\la+\mu}}}b_{\GGG,\kkk}^{\la,\mu} L_{G_1,k_1}^{\la,\ga}f_1(x)\Big|.
\end{align*}
Then it follows from H\"older's inequality and the maximal inequality for $\mathcal{M}$ that \eqref{LHSterm} is bounded by 
\begin{align*}
 2^{\la n(m-1)/2}\bigg( \prod_{j=2}^{m}\Vert f_j\Vert_{L^{p_j}(\bbrn)}\bigg) \sum_{\kkk^{*1}\in\PP_{*1}\UU_1^{\la+\mu}}  \bigg\Vert \Big( \sum_{\ga\in\bbz}\Big|  \sum_{k_1\in Col_{\kkk^{*1}}^{\UU_1^{\la+\mu}}}b_{\GGG,\kkk}^{\la,\mu} L_{G_1,k_1}^{\la,\ga}f_1   \Big|^2\Big)^{1/2}\bigg\Vert_{L^{p_1}(\bbrn)}.
\end{align*}
Now let $0<\epsilon_0<1$ be a sufficiently small number to be chosen later. 
Then Lemma \ref{keylemma3}, together with (\ref{maininftyest}) and (\ref{mainlqest}), yields that
\begin{align*}
& \bigg\Vert \Big( \sum_{\ga\in\bbz}\Big|  \sum_{k_1\in Col_{\kkk^{*1}}^{\UU_1^{\la+\mu}}}b_{\GGG,\kkk}^{\la,\mu} L_{G_1,k_1}^{\la,\ga}f_1   \Big|^2\Big)^{1/2}\bigg\Vert_{L^{p_1}(\bbrn)}\\
&\lesssim 2^{\la n/2}(\la+\mu+4)\Vert \Omega\Vert_{L^2(\mathbb{S}^{mn-1})}2^{-\delta\mu\epsilon_0}2^{-\la(L+1+mn)\epsilon_0}2^{(\la+\mu)n\epsilon_0}\Vert f_1\Vert_{L^{p_1}(\bbrn)}\\
&\lesssim 2^{\la n/2}2^{\mu \epsilon_0 (n-\delta)}2^{-\la \epsilon_0(L+1+mn-n)}(\la+\mu+4)\Vert \Omega\Vert_{L^2(\mathbb{S}^{mn-1})}\Vert f_1\Vert_{L^{p_1}(\bbrn)}
\end{align*} 
 as the cardinality of $Col_{\kkk^{*1}}^{\UU_1^{\la+\mu}}$ is less than $2^{(\la+\mu)n}$.
Finally, we have
\begin{equation*}
\Big\Vert  \sum_{\ga\in\bbz} \TT_{\GGG,1}^{\la,\ga,\mu}\big(f_1,\dots,f_m\big) \Big\Vert_{L^{p}(\bbrn)}\lesssim_{M_0} 2^{\epsilon \mu}2^{-\la M_0}\Vert \Omega\Vert_{L^2(\mathbb{S}^{mn-1})}\prod_{j=1}^{m}\Vert f_j\Vert_{L^{p_j}(\bbrn)},
\end{equation*} 
by choosing $\epsilon_0$ and $L$ such that
$$\epsilon=\epsilon_0n,\q M_0<\epsilon_0(L+1+mn-n)-mn/2.$$
This shows \eqref{2mmaingoal} for the case $l=1$.

\hfill

Now we suppose that $2\le l\le m$.
Using \eqref{disjointdecom1} and \eqref{lmaximalbound}, we write
\begin{align*}
 &\big| \TT_{\GGG,l}^{\la,\ga,\mu}\big(f_1,\dots,f_m\big)(x) \big|\\
 &\lesssim 2^{\la n(m-l)/2} \sum_{\kkk^{*1,\dots,l}\in \PP_{*1,\dots,l}\UU_l^{\la+\mu}}\Big( \prod_{j=l+1}^{m}\big| \mathcal{M}f_j(x)\big|   \Big)\bigg| \sum_{\kkk^{1,\dots,l}\in Col_{\kkk^{*1,\dots,l}}^{\UU_l^{\la+\mu}}} b_{\GGG,\kkk}^{\la,\mu} \prod_{j=1}^{l}L_{G_j,k_j}^{\la,\ga}f_j(x)    \; \bigg|
\end{align*} 
 and thus it follows from H\"older's inequality and the maximal inequality for $\mathcal{M}$ that  the left-hand side of (\ref{2mmaingoal}) is less than
 \begin{align*}
 &2^{\la n(m-l)/2} \Big( \prod_{j=l+1}^{m}\big\Vert f_j\big\Vert_{L^{p_j}(\bbrn)}\Big) \\
 &\qq\times \bigg( \sum_{\kkk^{*1,\dots,l}\in \PP_{*1,\dots,l}\UU_l^{\la+\mu}} \bigg\Vert  \sum_{\ga\in\bbz}   \Big|  \sum_{\kkk^{1,\dots,l}\in Col_{\kkk^{*1,\dots,l}}^{\UU_l^{\la+\mu}}} b_{\GGG,\kkk}^{\la,\mu} \prod_{j=1}^{l}L_{G_j,k_j}^{\la,\ga}f_j  \Big| \bigg\Vert_{L^{q_l}(\bbrn)}^{\min{\{1,p\}}}    \bigg)^{1/\min{\{1,p\}}}
 \end{align*}
 where $1/q_l:=1/p_1+\dots+1/p_l$.
Note that $Col_{\kkk^{*1,\dots,l}}^{\UU_l^{\la+\mu}}$ is a subset of $ (\WW^{\la+\mu})^l$ and thus
$$\big| Col_{\kkk^{*1,\dots,l}}^{\UU_l^{\la+\mu}}\big|\lesssim 2^{(\la+\mu) n l}.$$
Accordingly, 
Lemma \ref{keylemma32}, (\ref{maininftyest}), and (\ref{mainlqest}) yields that
 \begin{align*}
 &\bigg\Vert  \sum_{\ga\in\bbz}   \Big|  \sum_{\kkk^{1,\dots,l}\in Col_{\kkk^{*1,\dots,l}}^{\UU_l^{\la+\mu}}} b_{\GGG,\kkk}^{\la,\mu} \prod_{j=1}^{l}L_{G_j,k_j}^{\la,\ga}f_j  \Big| \bigg\Vert_{L^{q_l}(\bbrn)}\\
 &\lesssim  \Vert \Omega\Vert_{L^2(\mathbb{S}^{mn-1})}2^{\mu \epsilon_0(n-\delta)}2^{\la n l/2}(\la+\mu+4)^{l/\min{\{1,q_l\}}}2^{-\la \epsilon_0 (L+1+mn-nl)}\prod_{j=1}^{l}\Vert f_j\Vert_{L^{p_j}(\bbrn)}\\
 &\lesssim  \Vert \Omega\Vert_{L^2(\mathbb{S}^{mn-1})} 2^{\epsilon \mu} 2^{-\la (M_0+n(m-l)/2)}\prod_{j=1}^{l}\Vert f_j\Vert_{L^{p_j}(\bbrn)}
 \end{align*} 
 choosing $0<\epsilon_0<1$ and $L>0$ so that 
 $$ \epsilon=\epsilon_0 n \q \text{and} \q M_0+mn/2<\epsilon_0(L+1).$$
 This concludes that (\ref{2mmaingoal}) holds for $2\le l\le m$.
\end{proof}

\hfill

\begin{proof}[Proof of Proposition \ref{inductionpropo}]
Let $0<s<1$.
We first note that 
the direction $$\mathrm{ \bf Claims}~ X(s) \text{ and } Y(s)\Rightarrow \mathrm{\bf Claim}~Z(s)$$
follows from the (linear) Marcinkiewicz interpolation method. Here, we apply the interpolation separately $m$ times and in each interpolation, $m-1$ parameters among $p_1,\dots,p_m$ are fixed.
Moreover, the direction 
$$\mathrm{\bf Claim}~Z(s) \Rightarrow \mathrm{\bf Claim}~\Sigma(s)$$
also holds due to  Lemmas \ref{interpollemma} and \ref{convexhull}.

Therefore we need to prove the remaining direction
$\mathrm{ \bf Claim}~ X(s) \Rightarrow  \mathrm{ \bf Claim}~ Y(s)$.
The proof is based on the idea in \cite{He_Park2021}.
We are only concerned with the case $(1/p_1,\dots,1/p_m)\in\mathscr{R}_1^m(s)$ as a symmetric argument is applicable to the other cases.
 Assume that $p_1=1$, ${1}/{s}<p_2,\dots,p_m< \infty$, and $1+1/p_2+\cdots+1/p_m=1/p$.
Without loss of generality, we may also assume $\Vert f_1\Vert_{L^1(\bbrn)}=\Vert f_2\Vert_{L^{p_2}(\bbrn)}=\cdots=\Vert f_m\Vert_{L^{p_m}(\bbrn)}=\Vert \Omega\Vert_{L^{\frac{1}{1-s}}(\mathbb{S}^{mn-1})}=1$ and then it is enough to prove
\begin{equation}\label{weakmainest}
\Big|\Big\{    x\in \bbrn:    \big|\LL_{\mu}(f_1,\dots,f_m)(x)\big|>t    \Big\}\Big|\lesssim_{\epsilon} 2^{\epsilon \mu  p }{{t^{-p}}}.
\end{equation}
 We shall use the Calder\'on-Zygmund decomposition of $f_1$ at height $t^p$. Then $f_1$ can be expressed as
 $$f_1=g_1+\sum_{Q\in\mathcal{A}}{b_{1,Q}}$$
 where $\mathcal{A}$ is a subset of disjoint dyadic cubes, $\big| \bigcup_{Q\in \mathcal{A}}Q\big|\lesssim {t^{-p}}$, $\supp(b_{1,Q})\subset Q$, $\int{b_{1,Q}(y)}dy=0$, $\Vert b_{1,Q}\Vert_{L^1(\bbrn)}\lesssim t^{p}|Q|$, and $\Vert g_1\Vert_{L^r(\bbrn)}\lesssim t^{(1-{1}/{r})p}$ for all $1\le r\le \infty$.
Then the left-hand side of (\ref{weakmainest}) is less than
\begin{align*}
&\Big|\Big\{    x\in \bbrn:    \big|\LL_{\mu}(g_1,f_2,\dots,f_m)(x)\big|>t/2    \Big\}\Big|\\
&\qq +\bigg|\bigg\{    x\in \bbrn:    \Big|\LL_{\mu}\Big(\sum_{Q\in\mathcal{A}}{b_{1,Q}},f_2,\dots,f_m\Big)(x)\Big|>t/2    \bigg\}\bigg|=:\Xi_1^{\mu}+\Xi_2^{\mu}
\end{align*}

For the estimation of the first term, we choose ${1}/{s}<p_0<\infty$ and $\wt{p}>p$ with $1/p_0+1/p_2+\dots+1/p_m=1/\wt{p}$ 
and set $\epsilon_0:=\epsilon p/\wt{p}$ so that $0<\epsilon_0<1$.
Then the assumption $\mathrm{ \bf Claim}~ X(s) $ yields that
\begin{equation}\label{claimxnu}
\big\Vert \LL_{\mu}(g_1,f_2,\dots,f_m)\big\Vert_{L^{\wt{p}}(\bbrn)}\lesssim_{\epsilon_0}2^{\epsilon_0 \mu} \Vert g_1\Vert_{L^{p_0}(\bbrn)}\lesssim 2^{\epsilon_0\mu}t^{(1-1/p_0)p}.
\end{equation}
Now, using Chebyshev's inequality and the estimate (\ref{claimxnu}), the first term $\Xi_1^{\mu}$ is clearly dominated by
\begin{equation*}
 {t^{-\wt{p}}}\big\Vert \LL_{\mu}(g_1,f_2,\dots,f_m)\big\Vert_{L^{\wt{p}}(\bbrn)}^{\wt{p}}\lesssim 2^{\epsilon_0 \mu \wt{p}}t^{-\wt{p}(1-p(1-1/p_0))}= 2^{\epsilon \mu p}{t^{-p}}
\end{equation*}
since $\wt{p}(1-p(1-1/p_0))=p$.

Moreover, the remaining term $\Xi_2^{\mu}$ is estimated by the sum of $\big| \bigcup_{Q\in\mathcal{A}}Q^*\big|$ and
 \begin{equation*}
\Gamma_{\mu}:= \bigg| \bigg\{x\in \Big( \bigcup_{Q\in\mathcal{A}}Q^*\Big)^c : \Big| \LL_{\mu}\Big(\sum_{Q\in\mathcal{A}}b_{1,Q},f_2,\dots,f_m\Big)(x)\Big|>{t}/{2} \bigg\}\bigg|
 \end{equation*}
 where $Q^*$ is the concentric dilate of $Q$ with $\ell(Q^*)=10^2\sqrt{n}\ell(Q)$.
 Since  $\big| \bigcup_{Q\in\mathcal{A}}Q^*\big|\lesssim {t^{-p}}$, the proof of (\ref{weakmainest}) can be reduced to the inequality
 \begin{equation}\label{toshowprop}
 \Ga_{\mu} \lesssim_{\epsilon}2^{\epsilon \mu p}{t^{-p}}.
 \end{equation}
We apply Chebyshev's ineqaulity to deduce 
 \begin{align*}
 \Ga_{\mu}&\lesssim {t^{-p}}\int_{(\bigcup_{Q\in\mathcal{A}}Q^*)^c}\Big( \sum_{Q\in\mathcal{A}}\sum_{\gamma\in\bbz} \big|T_{K_{\mu}^{\ga}}\big(b_{1,Q},f_2,\dots,f_m \big)(x) \big|\Big)^{p}dx\\
 &\le {t^{-p}}\int_{(\bigcup_{Q\in\mathcal{A}}Q^*)^c}\Big( \sum_{Q\in\mathcal{A}}\sum_{\gamma: 2^{\ga}\ell(Q)\ge 1} \big|T_{K_{\mu}^{\ga}}\big(b_{1,Q},f_2,\dots,f_m \big)(x) \big|\Big)^{p}dx\\
 &\qq \qq \qq+{t^{-p}}\int_{\bbrn}\Big( \sum_{Q\in\mathcal{A}}\sum_{\gamma: 2^{\ga}\ell(Q)<1} \big|T_{K_{\mu}^{\ga}}\big(b_{1,Q},f_2,\dots,f_m \big)(x) \big|\Big)^{p}dx\\
 &=:\Ga_{\mu}^1+\Ga_{\mu}^2
 \end{align*}
 where $T_{K_{\mu}^{\ga}}$ is the multilinear operator associated with the kernel $K_{\mu}^{\ga}$ so that
 \begin{equation*}
  T_{K_{\mu}^{\ga}}\big( b_{1,Q},f_2,\dots,f_m\big)(x)=\int_{(\bbrn)^m} K_{\mu}^{\ga}(x-y_1,\dots,x-y_m)b_{1,Q}(y_1)\prod_{j=2}^{m}f_j(y_j) \;d\yyy.
  \end{equation*}
  
 To estimate $\Ga_{\mu}^1$, we see that
 \begin{align*}
 &\big|T_{K_{\mu}^{\ga}}\big(b_{1,Q},f_2,\dots,f_m \big)(x) \big| \\
 &\lesssim\int_{(\bbrn)^m}\int_{|\zzz|\sim 2^{-\ga}}2^{\ga mn}\big| \Omega(\zzz\; ')\big| \big|\Phi_{\mu+\ga}(x-y_1-z_1,\dots,x-y_m-z_m) \big|\\
 &\qq\qq\qq\qq\qq\qq\qq\qq\qq\times |b_{1,Q}(y_1)| \Big(\prod_{j=2}^{m}|f_j(y_j)|\Big)\; d\zzz \; d\yyy\\
 &\lesssim_L\int_{|\zzz|\sim 2^{-\ga}} 2^{\ga mn}\big| \Omega(\zzz\; ')\big| \Big(\int_{y_1\in Q}\frac{2^{(\mu+\ga)n}}{(1+2^{\mu+\ga}|x-y_1-z_1|)^L}|b_{1,Q}(y_1)|   dy_1\Big)\\
 & \qq\qq\qq\qq\qq\times \prod_{j=2}^{m}\Big(\int_{\bbrn}\frac{2^{(\mu+\ga)n}}{(1+2^{\mu+\ga}|x-y_j-z_j|)^L}|f_j(y_j)|   dy_j\Big) \;d\zzz
 \end{align*} for all $L>n$.
 Clearly,  we have
 \begin{equation}\label{maximalbound}
 \int_{\bbrn}\frac{2^{(\mu+\ga)n}}{(1+2^{\mu+\ga}|x-y_j-z_j|)^L}|f_j(y_j)|   dy_j\lesssim \mathcal{M}f_j(x-z_j), \qq j=2,\dots,m
 \end{equation} 
 and for $2^{\ga}\ell(Q)\ge 1$ and $|z_1|\le 2^{-\ga+1}$,
 $$\int_{y_1\in Q}\frac{2^{(\mu+\ga)n}}{(1+2^{\mu+\ga}|x-y_1-z_1|)^L}|b_{1,Q}(y_1)|   dy_1\lesssim \frac{2^{(\mu+\ga)n}}{(1+2^{\mu+\ga}|x-c_Q|)^L}\Vert b_{1,Q}\Vert_{L^1(\bbrn)}$$ because $|x-y_1-z_1|\gtrsim |x-c_{Q}|$.
 Therefore, we have
 \begin{align*}
 &\big|T_{K_{\mu}^{\ga}}\big(b_{1,Q},f_2,\dots,f_m \big)(x) \big|\\
 &\lesssim \frac{2^{(\mu+\ga)n}}{(1+2^{\mu+\ga}|x-c_Q|)^L}\Vert b_{1,Q}\Vert_{L^1(\bbrn)} \int_{|\zzz|\sim 2^{-\ga}} 2^{\ga mn}\big| \Omega(\zzz \; ')\big| \Big( \prod_{j=2}^{m}\mathcal{M}f_j(x-z_j)\Big) d\zzz.
 \end{align*}
Now H\"older's inequality yields
 \begin{align*}
 &\int_{|\zzz|\sim 2^{-\ga}} 2^{\ga mn}\big| \Omega(\zzz \; ')\big| \Big( \prod_{j=2}^{m}\mathcal{M}f_j(x-z_j)\Big)\; d\zzz \nonumber\\
 &\le \Big(\int_{|\zzz|\sim 2^{-\ga}}{2^{\ga mn}\big| \Omega(\zzz\; ')\big|^{\frac{1}{1-s}}}d\zzz \Big)^{1-s}\Big( \int_{|\zzz|\sim 2^{-\ga}}{2^{\ga mn} \Big( \prod_{j=2}^{m}\mathcal{M}f_j(x-z_j)\Big)^{\frac{1}{s}}}d\zzz \;'\Big)^{s} \nonumber\\
 &\le \Vert \Omega\Vert_{L^{\frac{1}{1-s}}(\mathbb{S}^{mn-1})}\prod_{j=2}^{m}\Big(2^{\ga n}\int_{|z_j|\lesssim 2^{-\ga}}\big| \mathcal{M}f_j(x-z_j)\big|^{\frac{1}{s}}dz_j \Big)^{s}\lesssim \prod_{j=2}^{m}\mathcal{M}_{\frac{1}{s}}\mathcal{M}f_j(x)
 \end{align*}
 and thus
 \begin{align*}
 \big|T_{K_{\mu}^{\ga}}\big(b_{1,Q},f_2,\dots,f_m \big)(x) \big|\lesssim  \frac{2^{(\mu+\ga)n}}{(1+2^{\mu+\ga}|x-c_Q|)^L}\Vert b_{1,Q}\Vert_{L^1(\bbrn)}\prod_{j=2}^{m}\mathcal{M}_{\frac{1}{s}}\mathcal{M}f_j(x).
 \end{align*}
 This, together with H\"older's inequality, deduces that $ \Ga_{\mu}^1$ is dominated by a constant times
 \begin{align*}
&{t^{-p}}\int_{(\bigcup_{Q\in\mathcal{A}}Q^*)^c} \Big( \prod_{j=2}^{m} \mathcal{M}_{\frac{1}{s}}\mathcal{M}f_j(x)\Big)^p\Big(\sum_{Q\in\mathcal{A}}\sum_{\ga: 2^{\ga}\ell(Q)\ge 1}  \frac{2^{(\mu+\ga)n}}{(1+2^{\mu+\ga}|x-c_Q|)^L}\Vert b_{1,Q}\Vert_{L^1(\bbrn)}     \Big)^{p}dx\\
 &\le {t^{-p}} \Big( \prod_{j=2}^{m}\big\Vert \mathcal{M}_{\frac{1}{s}}\mathcal{M}f_j\big\Vert_{L^{p_j}(\bbrn)} \sum_{Q\in\mathcal{A}}\sum_{\ga: 2^{\ga}\ell(Q)\ge 1}\Big\Vert \frac{2^{(\mu+\ga)n}}{(1+2^{\mu+\ga}|\cdot-c_Q|)^L} \Big\Vert_{L^1((Q^*)^c)}\Vert b_{1,Q}\Vert_{L^1(\bbrn)}\Big)^{p}.
 \end{align*}
 Since ${1}/{s}<p_2,\dots,p_m<\infty$, 
 each $L^{p_j}$ norm is controlled by $\Vert f_j\Vert_{L^{p_j}(\bbrn)}=1$, using the $L^{p_j}$ boundedness of both $\mathcal{M}_{\frac{1}{s}}$ and $\mathcal{M}$.
 Moreover, using the fact that for $2^{\mu-10}\ge C_0\sqrt{mn}$,
 $$\Big\Vert \frac{2^{(\mu+\ga)n}}{(1+2^{\mu+\ga}|\cdot-c_Q|)^L} \Big\Vert_{L^1((Q^*)^c)}\lesssim 2^{-\mu(L-n)}\big(2^{\ga}\ell(Q) \big)^{-(L-n)}\le \big(2^{\ga}\ell(Q) \big)^{-(L-n)},$$ we have
 \begin{equation*}
 \sum_{Q\in\mathcal{A}}\sum_{\ga: 2^{\ga}\ell(Q)\ge 1}\Big\Vert \frac{2^{(\mu+\ga)n}}{(1+2^{\mu+\ga}|\cdot-c_Q|)^L} \Big\Vert_{L^1((Q^*)^c)}\Vert b_{1,Q}\Vert_{L^1(\bbrn)}\lesssim  \sum_{Q\in\mathcal{A}}\Vert b_{1,Q}\Vert_{L^1(\bbrn)}\lesssim 1.
 \end{equation*}
 This concludes 
 $$\Ga_{\mu}^1\lesssim t^{-p}.$$
 
Next, let us deal with the other term $\Ga_{\mu}^2$.
 By using the vanishing moment condition of $b_{1,Q}$, we have
 \begin{align}\label{tkest}
 &\big| T_{K_{\mu}^{\ga}}\big(b_{1,Q},f_2,\dots,f_m\big)(x)\big|\nonumber\\
 &\lesssim \int_{|\zzz|\sim 2^{-\ga}} 2^{\ga mn}\big| \Omega(\zzz ')\big| \bigg(\int_{(\bbrn)^m}\big| \Phi_{\mu+\ga}(x-y_1-z_1,\dots,x-y_m-z_m)\\
 & \qq -\Phi_{\mu+\ga}(x-c_Q-z_1,x-y_2-z_2,\dots,x-y_m-z_m)\big| \big|b_{1,Q}(y_1)\big| \Big( \prod_{j=2}^{m}|f_j(y_j)|\Big)d\yyy \bigg) \; d\zzz. \nonumber
 \end{align}
 We observe that
 \begin{align*}
&\big| \Phi_{\mu+\ga}(x-y_1-z_1,\dots,x-y_m-z_m)-\Phi_{\mu+\ga}(x-c_Q-z_1,x-y_2-z_2,\dots,x-y_m-z_m)\big| \\
&\lesssim 2^{(\mu+\ga)}\ell(Q)\; V^L_{\mu+\ga}(x-z_1,y_1,c_Q)\;\Big(\prod_{j=2}^{m}\frac{2^{(\mu+\ga)n}}{(1+2^{\mu+\ga}|x-y_j-z_j|)^L}\Big)
 \end{align*}
 where
 $$V^L_{\mu+\ga}(x,y_1,c_Q):=\int_0^1\frac{2^{(\mu+\ga)n}}{(1+2^{\mu+\ga}|x-ty_1-(1-t)c_Q|)^L}dt.$$
Furthermore,
 \begin{align*}
 &\big| \Phi_{\mu+\ga}(x-y_1-z_1,\dots,x-y_m-z_m)-\Phi_{\mu+\ga}(x-c_Q-z_1,x-y_2-z_2,\dots,x-y_m-z_m)\big|\\
 &\lesssim_L W_{\mu+\ga}^L(x-z_1,y_1,c_Q)\;\Big(\prod_{j=2}^{m}\frac{2^{(\mu+\ga)n}}{(1+2^{\mu+\ga}|x-y_j-z_j|)^L}\Big)
 \end{align*}
 where $$W_{\mu+\ga}^L(x,y_1,c_Q):=\frac{2^{(\mu+\ga)n}}{(1+2^{\mu+\ga}|x-y_1|)^L}+\frac{2^{(\mu+\ga)n}}{(1+2^{\mu+\ga}|x-c_Q|)^L}.$$
By averaging these two estimates and letting
$$U_{\mu+\ga}^{L,\epsilon}(x,y_1,c_Q):=\big( V_{\mu+\ga}^L(x,y_1,c_Q)\big)^{\epsilon}\big(W_{\mu+\ga}^{L}(x,y_1,c_Q) \big)^{1-\epsilon},$$
 we obtain 
\begin{align}\label{differenceest}
&\big| \Phi_{\mu+\ga}(x-y_1-z_1,\dots,x-y_m-z_m)-\Phi_{\mu+\ga}(x-c_Q-z_1,x-y_2-z_2,\dots,x-y_m-z_m)\big| \nonumber\\
&\lesssim_{L,\epsilon}2^{\epsilon \mu}\big(2^{\ga}\ell(Q) \big)^{\epsilon}    U_{\mu+\ga}^{L,\epsilon}(x-z_1,y_1,c_Q)    \;\Big(\prod_{j=2}^{m}\frac{2^{(\mu+\ga)n}}{(1+2^{\mu+\ga}|x-y_j-z_j|)^L}\Big).
\end{align}
 Here, we note that
 \begin{equation*}
 \big\Vert U_{\mu+\ga}^{L,\epsilon}(\cdot, y_1,c_Q)\big\Vert_{L^1(\bbrn)}\le \big\Vert V_{\mu+\ga}^{L}(\cdot,y_1,c_Q)\big\Vert_{L^1(\bbrn)}^{\epsilon} \big\Vert W_{\mu+\ga}^{L}(\cdot,y_1,c_Q)\big\Vert_{L^1(\bbrn)}^{1-\epsilon}\lesssim 1.
 \end{equation*}
 By plugging (\ref{differenceest}) into (\ref{tkest}), we obtain
 \begin{align*}
 &\big| T_{K_{\mu}^{\ga}}\big(b_{1,Q},f_2,\dots,f_m\big)(x)\big|\\
 & \lesssim 2^{\ep \mu}\big( 2^{\ga}\ell(Q)\big)^{\ep}\int_{|\zzz|\sim 2^{-\ga}}2^{\ga m n}\big| \Omega(\zzz ' )\big|\Big(\int_{\bbrn}{U_{\mu+\ga}^{L,\ep}(x-z_1,y_1,c_Q) |b_{1,Q}(y_1)|}dy_1 \Big)\\
 &\qq\qq \qq\qq \qq \times \Big(\prod_{j=2}^{m}\int_{\bbrn}{\frac{2^{(\mu+\ga)n}}{(1+2^{\mu+\ga}|x-y_j-z_j|)^L} |f_j(y_j)|}dy_j \Big)\; d\zzz\\
 &\lesssim 2^{\ep \mu}\big( 2^{\ga}\ell(Q)\big)^{\ep}
 \int_{|z_1|\lesssim 2^{-\ga}}\int_{\bbrn}{U_{\mu+\ga}^{L,\ep}(x-z_1,y_1,c_Q) |b_{1,Q}(y_1)|}dy_1\\
 &\qq\qq \qq\qq \qq \times \Big(\int_{|(z_2,\dots, z_m)|\lesssim 2^{-\ga}}2^{\ga  mn}\big| \Omega(\zzz' )\big|\prod_{j=2}^m\mathcal{M}f_j(x-z_j)\; dz_2\cdots dz_m\Big) \;dz_1
 \end{align*}
    where (\ref{maximalbound}) is applied. The innermost integral is, via H\"older's inequality, bounded by
    $$2^{\ga n}\Big( \int_{|(z_2,\dots,z_m)|\lesssim 2^{-\ga}} 2^{\ga (m-1)n}\big| \Omega(\zzz')\big|^{\frac{1}{1-s}}dz_2\cdots dz_m\Big)^{1-s}\prod_{j=2}^{m}\mathcal{M}_{\frac{1}{s}}\mathcal{M}f_j(x)$$
and thus we have
 \begin{align*}
& \big| T_{K_{\mu}^{\ga}}\big(b_{1,Q},f_2,\dots,f_m\big)(x)\big|  \lesssim 2^{\ep \mu}\big( 2^{\ga}\ell(Q)\big)^{\ep} \;\prod_{j=2}^m\mathcal{M}_{\frac{1}{s}}\mathcal{M}f_j(x)\int_{\bbrn} |b_{1,Q}(y_1)|
\\ 
 &\q \times  \int_{|z_1|\lesssim 2^{-\ga}}{2^{\ga n}U_{\mu+\ga}^{L,\ep}(x-z_1,y_1,c_Q) }
\Big(\int_{|(z_2,\dots,z_m)|\lesssim 2^{-\ga}}2^{\ga (m-1)n}\big| \Omega(\zzz' )\big|^{\tf1{1-s}}\; dz_2\cdots dz_m\Big)^{1-s} dz_1dy_1
  \end{align*}
 Now, by using H\"older's inequality and the maximal inequality for $\mathcal{M}_{\frac{1}{s}}$ and $\mathcal{M}$, we have
\begin{align*}
\Ga_{\mu}^2 &\lesssim t^{-p}2^{\epsilon \mu p}\bigg\Vert    \sum_{Q\in\mathcal{A}}\sum_{\gamma: 2^{\ga}\ell(Q)<1} \big( 2^{\ga}\ell(Q)\big)^{\ep}
\int_{\bbrn} |b_{1,Q}(y_1)|   \int_{|z_1|\lesssim 2^{-\ga}}2^{\ga n}{U_{\mu+\ga}^{L,\ep}(\cdot-z_1,y_1,c_Q) }\\
&\qq\qq\times \Big(\int_{|(z_2,\dots, z_m)|\lesssim 2^{-\ga}}2^{\ga (m-1)n}\big| \Omega(\zzz' )\big|^{\tf1{1-s}}\; dz_2\cdots dz_m\Big)^{1-s} dz_1dy_1   \bigg\Vert_{L^1(\bbrn)}^p.
\end{align*}
Moreover, the $L^1$ norm in the last displayed expression is bounded by
  \begin{align*}
&\sum_{Q\in\mathcal{A}}\sum_{\ga:2^{\ga}\ell(Q)<1}\big( 2^{\ga}\ell(Q)\big)^{\ep} \int_{\bbrn} |b_{1,Q}(y_1)|\big\Vert U_{\mu+\ga}^{L,\ep}(\cdot,y_1,c_Q) \big\Vert_{L^1(\bbrn)}\\
&\qq\times  \int_{|z_1|\lesssim 2^{-\ga}}2^{\ga n}\Big(\int_{|(z_2,\dots, z_m)|\lesssim 2^{-\ga}}2^{\ga (m-1)n}\big| \Omega(\zzz' )\big|^{\tf1{1-s}}\; dz_2\cdots dz_m\Big)^{1-s} dz_1 \;dy_1\\
&\lesssim  \sum_{Q\in\mathcal{A}}\sum_{\ga:2^{\ga}\ell(Q)<1}\big( 2^{\ga}\ell(Q)\big)^{\ep}\int_{\bbrn}|b_{1,Q}(y_1)|\Big( \int_{|\zzz|\lesssim 2^{-\ga}}2^{\ga mn}\big| \Omega(\zzz')\big|^{\frac{1}{1-s}} d\zzz\Big)^{1-s}dy_1\\
&\lesssim \sum_{Q\in\mathcal{A}}\Vert b_{1,Q}\Vert_{L^1(\bbrn)}    \sum_{\ga:2^{\ga}\ell(Q)<1}\big( 2^{\ga}\ell(Q)\big)^{\ep}\lesssim 1
\end{align*}
where the first inequality follows from H\"older's inequality.
This proves 
$$\Ga_{\mu}^2\lesssim t^{-p}2^{\epsilon \mu p},$$
which finally completes the proof of \eqref{toshowprop}.
\end{proof}



\end{document}